\documentclass[a4paper,12pt,oneside]{article}
\usepackage[margin=1in]{geometry}

\usepackage{amsmath,amsthm,amssymb}
\usepackage{units}
\usepackage{enumitem}
\usepackage{mathrsfs}
\usepackage{hyperref}
\allowdisplaybreaks


\newtheorem{theorem}{Theorem}
\newtheorem{lemma}[theorem]{Lemma}
\newtheorem{proposition}[theorem]{Proposition}

\theoremstyle{definition}
\newtheorem{definition}[theorem]{Definition}
\newtheorem{remark}[theorem]{Remark}

\begin{document}

\title{Length two extensions of modules for the Witt Algebra}
\date{}
\author{Kathlyn Dykes}
\maketitle

\begin{abstract} 
In this paper, we establish an explicit classification of length two extensions of tensor modules for the Witt algebra using the cohomology of the Witt algebra with coefficients in the module of the space of homomorphisms between the two modules of interest. To do this we extended our module to a module that has a compatible action of the commutative algebra of Laurent polynomials in one variable. In this setting, we are be able to directly compute all possible 1-cocycles.
\end{abstract}

\section{Introduction}

This paper is concerned with classifying all length two extensions of tensor modules of the Witt algebra, $W_1$. This is a question that has been studied before by Martin and Piard in their 1992 paper \ref{mp}. In \ref{mp}, Martin and Piard classifies all the cases in which these extensions exist and describe the action of the generators of the Witt algebra. In this paper, we construct much better bases for these extensions which allow us to give explicit formulas for the action of the Witt algebra.
 By comparison, our results are similar to the classification of the cohomology of the Lie algebra of vector fields on a line given by Feigin and Fuks in their 1982 paper \ref{FF}.

In his 1992 paper \ref{Mathieu}, Mathieu gave a complete classification of irreducible Virasoro modules with finite dimensional weight spaces. Consequently, as any Witt algebra module is a Virasoro module where the central element acts trivially, his results are also applicable to Witt algebra modules. As was proved in \ref{Mathieu}, every irreducible $W_1$-module with finite-dimensional weight spaces is either a highest weight modules, a lowest weight module, a tensor module or quotient of a tensor module. 

In this paper, the method used to obtain the classification of these module extensions $M$, begins with moving up to the $A$-cover $\widehat{M}$, as described in \ref{W_nmodules}, where $A$ is the algebra of Laurent polynomials in one variable. The idea is to obtain a classification in the context of $AW_1$-modules, which is a simpler problem computationally than in the $W_1$-module setting. We obtain a surjection from the space where the classification is easily calculated back to the original module, which will give a  classification of these extensions.
 In fact, we can extend this definition to define an $A_nW_n$-module \ref{W_nmodules}. The methods used in this work will extend to the case of $W_n$, giving us somewhere to start in the classification of length two extensions of $W_n$-modules.

The paper is organized as follows: In Section 2, we find when these extensions are polynomial. In Section 3, we classify all polynomial cocycles. In Section 4, we classify all delta function cocycles. In Section 5, we look at a special case that does not work under the general approach used in Section 2. Finally, we end with a summary of all results. 

\subsection{The Witt algebra and its irreducible modules}

In this paper, we are interested in the Witt algebra, $W_1$. We consider the algebra of Laurent polynomials over the complex numbers, denoted by $ A = \mathbb{C} [t^{\pm1}]$. The Witt algebra is the set of all derivations of $A$, which is a Lie algebra spanned by the basis vectors $\left\lbrace e_k = t^{k+1} \frac{d}{dt} | k \in \mathbb{Z} \right\rbrace$ with the bracket $[e_k, e_s]  = (s-k)e_{k+s}$.

We can consider tensor modules which are parameterized by two constants, $\gamma, \alpha \in \mathbb{C}$. The structure of these tensor $W_1$-modules, denoted by $T(\alpha, \gamma)$, are given by $$T(\alpha,\gamma) = \bigoplus_{m \in \gamma + \mathbb{Z}} \mathbb{C} v_m$$ with $W_1$ action
\[
 e_k v_m = (m + \alpha k) v_{m+k}
\]

Note that $T(\gamma, \alpha) \cong T(\gamma', \alpha)$ if $\gamma - \gamma' \in \mathbb{Z}$, as stated in Remark 1.1 of \ref{kac}. Therefore, we will always assume that $0 \leq \text{Re}(\gamma) < 1$ and we will treat the cases of $\gamma \in \mathbb{Z}$ and $\gamma =0$ as interchangeable.

These tensor modules are irreducible except in two special cases: $\alpha =0, \gamma \in \mathbb{Z}$ and $\alpha =1, \gamma \in \mathbb{Z}$ by Proposition 1.1 of \ref{kac}. For convenience, we will denote by $D(0) = \text{span}\lbrace v_0 \rbrace$ the irreducible submodule of $T(0,0)$ and $T^\circ(1) = \text{span}\lbrace v_m | m \in \mathbb{Z}, m \neq 0 \rbrace$ the irreducible submodule of $T(1,0)$.

\begin{definition} A $W_1$-module $V$ is a \emph{weight module} if there exist weights $\lambda \in \mathbb{C}$ such that $V = \oplus_{\lambda} V_\lambda$ where the weight space of weight $\lambda$ is given by
\[
V_\lambda = \left\lbrace v \in V | e_0 v = \lambda v \right\rbrace
\]
\end{definition}

\begin{definition}
 A weight module is called \emph{cuspidal} if the dimensions of all weight spaces are bounded by a common constant.
\end{definition}

Tensor $W_1$-modules are cuspidal modules as every weight space has dimension 1. In fact, every irreducible cuspidal module of $W_1$ is isomorphic to either a tensor module, $D(0)$ or $T^\circ (1)$, as the only cuspidal highest and lowest weight modules are trivial (Corollary III.3 of \ref{mp2}). 

\begin{definition}
If $V$ is both an $A$-module and a $W_1$-module, then the $A$-action and the $W_1$-action is \emph{compatible} if 
\begin{equation}
 (xf)v = x(fv) - f(xv), \text{ for all } x\in W_1, f \in A, v \in V
\end{equation}
If $V$ has a compatible $A$-action and $W_1$ action, it is called an \emph{$AW_1$-module} \ref{W_nmodules}.
\end{definition}

The space of $\text{Hom}_\mathbb{C}(A, V)$ is an example of $AW_1$-module  \ref{W_nmodules}, with the actions
\begin{align}\label{homAM}
(x \psi) (f) =& x(\psi(f)) - \psi(x(f)) \\
(g \psi)(f) =& \psi(gf), \text{ for } \psi \in \text{Hom}(A,V), x\in W_1, f,g \in A
\end{align}

In general, this space is too ``big'' in the sense that it may have infinite dimensional weight spaces. Instead, we would like to look at a submodule of Hom$(A,V)$ which will better preserve the structure of the weight spaces of the module $V$.

\begin{definition} If $V$ is a $W_1$-module, the $A$-cover $\widehat{V}$ is defined as
\[ \widehat{V} = \text{span}\lbrace \phi(x,u) | x \in W_1, u \in V \rbrace \subset \text{Hom}(A, V) \]
where $\phi(x,u)(f) = (fx)(u)$ for all $f \in A$. 
\end{definition}

The $W_1$-action and $A$-action on the $A$-cover is
\begin{align}\label{Acover}
 y\phi(x,u) =& \phi([y,x],u)+\phi(x,yu) \\
 g\phi(x,u) =& \phi(gx,u)
\end{align}
for $x,y \in W_1, u \in V, g \in A$. 

The following Theorem is important to the motivation behind considering the $A$-cover of a cuspidal module $V$.

\begin{theorem}[Theorem 4.10 of \ref{W_nmodules}]\label{cupsidal}
If $V$ is a cuspidal module, then so is $\widehat{V}$. 
\end{theorem}

From $\widehat{V}$, we can always go back to the space $V$ via the map $\pi: \widehat{V} \rightarrow V$ where $\pi(\phi(x,u)) = \phi(x,u)(1) = x.u$. By Proposition 4.5 of \ref{W_nmodules}, this map is surjective only if $W_1 V=V$.

\section{Cocycle Functions} 

In this paper, we will concern ourselves with weight $W_1$-modules $M$ which possess the short exact sequence:
\[
0 \rightarrow T(\alpha, \gamma) \rightarrow M \rightarrow  T(\beta, \gamma')\rightarrow 0
\]
These extensions are taken to be a weight extensions, i.e. the weight space $M_\lambda$ is given by the exact sequence
\[
0 \rightarrow T(\alpha, \gamma)_\lambda \rightarrow M_\lambda \rightarrow T(\beta, \gamma')_\lambda \rightarrow 0
\]

Here we fix a basis $\lbrace v_m \rbrace$ of $T(\alpha, \gamma)$ and a basis $\lbrace w_m \rbrace$ of $T(\beta,\gamma')$. In this way, $M$ is spanned by basis vectors $v_m$ and $w_m$, where these $w_m \in M$ are mapped to $w_m \in T(\beta,\gamma')$ under the surjection $\overline{\cdot} :M \rightarrow T(\beta,\gamma')$. 

The goal of this section is to make use of the A-cover to find an appropriate basis of $M$ so that the corresponding cocycles will be polynomial in almost all cases.

\subsection{Parameters of the module extension}

As $T(\alpha, \gamma)$ is a $W_1$-submodule of $M$, then the $W_1$ action is
\[
 e_k v_m = (m+ \alpha k) v_{m+k}, \text{ for } k\in \mathbb{Z}, m \in \gamma + \mathbb{Z}
\]
The $W_1$-action on the $w_m$ basis vectors is defined as
\begin{equation}\label{basicaction}
 e_k w_m = (m+ \beta k) w_{m+k} + \tau(k,m)v_{m+k}, \text{ for } k\in \mathbb{Z}, m \in \gamma' + \mathbb{Z}
\end{equation}
where $\tau(k,m)$ is some function in $k$ and $m$.

 When $\gamma + \mathbb{Z} \neq \gamma' + \mathbb{Z}$, this gives a trivial extension in the sense that $\tau(k,m)$ is the zero function. Therefore, we may assume these cosets are equal, so that without loss of generality, we will assume $\gamma = \gamma'$ for the rest of the paper. Here we see that $M$ is parameterized by four objects: the function $\tau(k,m)$ and the complex numbers $\alpha, \beta$ and $\gamma$.

As the $W_1$-action on $M$ is a module action, we obtain
 the following condition on $\tau$:
\begin{align}\label{taucondition}
\begin{split}
(s-k) \tau(k+s, m) =& (m + \beta s)\tau(k,m+s)-  (m + \beta k) \tau(s,m+k)
\\ & + (m+s + \alpha k )\tau(s,m) - (m+k+ \alpha s)\tau(k,m)
\end{split}
\end{align}
These $\tau$-functions also have a cohomological interpretation.

\begin{lemma}
For a Lie algebra $L$, let $M$ be the $L$-module extension of the $L$-modules $(M_1, \rho_1), (M_2, \rho_2)$ given by
\[
0 \rightarrow M_1 \rightarrow M \rightarrow M_2 \rightarrow 0
\]
The isomorphism classes of these module extensions $M$ are in one-to-one correspondence to the cohomology, $H^1 (L, \text{Hom}_\mathbb{C}(M_2, M_1))$.
\end{lemma}

For $W_1$, $\tau(k,m)$ is a 1-cocycle of the cohomology of $W_1$ with coefficients in the module Hom$(T(\beta, \gamma),T(\alpha, \gamma))$.
There is a natural equivalence relation on this set of 1-cocycles where two cocycles are equivalent if they differ by a 1-coboundary. In this way, we obtain a one-to-one correspondence between equivalence classes of 1-cocycles and equivalence classes of extensions $M$.

\subsection{Lifting the module extension into the setting of $AW_1$-modules}

Our final goal is to obtain a basis for $M$ such that our cocycles will be polynomials in almost all cases. By this, we mean we want to find $\widehat{w}_k \in M$ such that each $\widehat{w}_k$ is the preimage of $w_k \in T(\beta,\gamma)$ that admit polynomial cocycles $\tau$. The first thing to do is to lift our module extension into the setting of $AW_1$-modules. To do this, we make use of the $A$-cover $\widehat{M}$ of $M$.

From Theorem 4.10 of \ref{W_nmodules}, the $A$-cover of $M$ is cuspidal, so $\widehat{M}$ has finite dimensional weight spaces. We have the map $\pi: \widehat{M} \rightarrow M$ such that $\pi (\phi(x,u)) = \phi(x,u)(1) = x. u$. Define
 \[ \
 \widehat{M}(\alpha,\gamma) = \left\lbrace \sum_i \phi (x_i, u_i) | \forall f \in A, \sum_i (fx_i)u_i \in T(\alpha,\gamma) \right\rbrace. 
 \]
 and define the map $\widehat{\pi}: \widehat{M} \rightarrow \widehat{T(\beta,\gamma)}$ by $\widehat{\pi}(\phi(x,u)) = \phi(x,\overline{u})$, where $\overline{u}$ is the image of $u$ under the surjection that sends $M$ onto $T(\beta,\gamma)$.

\begin{lemma}\label{pihat}
$\widehat{\pi}$ is an homomorphism of $AW_1$-modules and $\widehat{M}/{ \widehat{M}(\alpha,\gamma)}\cong \widehat{T(\beta,\gamma)}$.
\end{lemma}

\begin{proof}

For the first part, as $M \rightarrow T(\beta,\gamma)$ is a homomorphism of $W_1$-modules, then it is enough to show that $\widehat{\pi}$ is a homomorphism of $A$-modules. For $f \in A, x \in W_1, u \in M$,
\[
 f \widehat{\pi}(\phi(x,u)) = f \phi(x,\bar{u}) = \phi(fx, \bar{u}) = \widehat{\pi} (\phi(fx, u))= \widehat{\pi}(f \phi(x,u) )
\]
so that the $A$-action is preserved by $\widehat{\pi}$.

The second part is clear, as $\widehat{M}(\alpha, \gamma)$ is defined exactly as it needs to be the kernel of this map.
\end{proof}

It follows from this lemma that $\widehat{M}(\alpha,\gamma)$ is an $AW_1$-module as it is the kernel of a homomorphism of $AW_1$-modules. We obtain the short exact sequence 
\[
0 \rightarrow \widehat{M}(\alpha,\gamma) \rightarrow \widehat{M} \rightarrow \widehat{T(\beta,\gamma)} \rightarrow 0
\]
This extension turns out to be too large; instead we will need to find a submodule  $\widetilde{M}$ of $\widehat{M}$ that will admit a short exact sequence with $T(\beta,\gamma)$;
the next two lemmas will help us relate $\widehat{T(\beta,\gamma)}$ to $T(\beta,\gamma)$. Define $\varepsilon_j$ and $\eta_j$ in Hom$(A, T(\beta, \gamma))$ by
\[
\varepsilon_j (t^m) = w_{m+j}, \, \, \,
\eta_j (t^m) = (m+j) w_{m+j}, \, \, \, j \in \gamma + \mathbb{Z}
\]

\begin{lemma}\label{epsilon_j}
The $A$-cover $\widehat{T(\beta, \gamma)}$ is spanned by the vectors $\varepsilon_j, \eta_j$ for $j\in \gamma + \mathbb{Z}$ when $\beta \neq 1$, and $\beta \neq 0$. In the special cases of $\beta=0$ or $\beta =1$, $\widehat{T(\beta, \gamma)}$ is spanned by vectors $\varepsilon_j$ or $\eta_j$ respectively. 
\end{lemma}

\begin{proof}
The action of $\phi(e_k, w_j)$ on $t^m$ is given by:
\begin{equation}\label{etaepsilon}
\phi(e_k, w_j)(t^m) = \beta \eta_{j+k}(t^m) + (1-\beta) j \varepsilon_{j+k}(t^m)
\end{equation}

For $\beta \neq 1$, consider
\[\dfrac{1}{1-\beta}\left( \phi(e_0, w_j ) - \phi(e_1, w_{j-1}) \right) \in \widehat{T(\beta, \gamma)}
\] 
 where $w_{j}, w_{j-1} \in T(\beta, \gamma), j \in \gamma + \mathbb{Z}$. Then this element maps $t^m$ to $w_{j+m}$  
so that, as long as $\beta \neq 1$, $\varepsilon_{j} \in \widehat{T(\beta, \gamma)}$ for all $j \in \gamma + \mathbb{Z}$.

For $\eta_j$, if $\beta \neq 0$, $\eta_{j} = \frac{1}{\beta} \left( \phi(e_k, w_{j-k})- (1-\beta) (j-k) \varepsilon_{j} \right)$ and so $\eta_{j} \in \widehat{T(\beta, \gamma)}$ as long as $\beta \neq 0 $.
\end{proof}

\begin{remark}
Here, $\widehat{T(\beta, \gamma)}$ is a module extension of $T(\beta -1, \gamma)$ with $T(\beta, \gamma)$ when $\beta$ is not $0$ or $1$. From the $W_1$-action on these basis vectors (calculated in (\ref{calculations}) and (\ref{cal2})) and the results of this paper, the cocycle of this extension is trivial so that $\widehat{T(\beta, \gamma)} \cong T(\beta, \gamma) \oplus T(\beta -1, \gamma)$.
\end{remark}

\begin{lemma}
Let $\beta \neq 1$. The subspace of $\widehat{T(\beta, \gamma)}$ spanned by $\lbrace \varepsilon_i | i \in \gamma+ \mathbb{Z} \rbrace$ is isomorphic to $T(\beta, \gamma)$.
\end{lemma}

\begin{proof}
The action of $W_1$ on $\varepsilon_m$ is given by $ e_k  \varepsilon_m = (m+\beta k) \varepsilon_{m+k}$ 
so that $\lbrace \varepsilon_i | i \in \gamma + \mathbb{Z} \rbrace$ is a submodule of $\widehat{T(\beta, \gamma)}$. Then, as long as $\beta \neq 1$, the map $\varepsilon_i \rightarrow w_i$ is surjective and the lemma follows.
\end{proof}

Consequently, we will denote span$\lbrace \varepsilon_j | j \in \mathbb{Z} \rbrace$ by $T(\beta,\gamma)$, and view $T(\beta,\gamma)$ as a subspace of $\widehat{T(\beta,\gamma)}$ when $\beta \neq 1$.

Now define $\widetilde{M} = \widehat{\pi} ^{-1}(T(\beta,\gamma))$. $\widetilde{M}$ is an $AW_1$-module containing $\widehat{M}(\alpha, \gamma)$. We now determine $\widetilde{M}/\widehat{M}(\alpha, \gamma)$ to write $\widetilde{M}$ as a module extension; this is the appropriate setting in which to view our cocycles. 

\begin{lemma}
 $\pi: \widetilde{M} \rightarrow M$ is surjective when $\beta \neq 1$.
\end{lemma}

\begin{proof}
 Recall that $ \pi: \widehat{T(\alpha,\gamma)} \rightarrow T(\alpha,\gamma)$ will be surjective if $W_1 T(\alpha,\gamma) = T(\alpha,\gamma)$.
 Notice $\frac{1}{k}e_0v_k =v_k$ for $k \neq 0$, so that this map is surjective when $\gamma \neq 0$. For $v_0 \in W_1 T(\alpha,0)$, $ \frac{1}{(-s+\alpha s)}e_s v_{-s} \in W_1 T(\alpha,0)$ which happens only if $\alpha \neq 1$. Thus, this map is surjective as long as $\alpha \neq 1$ or $\gamma \notin \mathbb{Z}$:

For $\pi: \widehat{T(\alpha,\gamma)} \rightarrow T(\alpha,\gamma)$ is surjective, we can extend to the map $\widehat{M}(\alpha, \gamma) \rightarrow T(\alpha,\gamma)$, which is also surjective. It is left to show that $\widehat{\pi}: \widetilde{M}/\widehat{M}(\alpha,\gamma) \rightarrow T(\beta,\gamma)$ is surjective.

By Lemma \ref{pihat}, $\widehat{\pi}: \widehat{M} \rightarrow \widehat{T(\beta, \gamma)}$ is surjective. As $T(\beta, \gamma) \subseteq \widehat{T(\beta, \gamma)}$ for $\beta \neq 1$ and $\widetilde{M} = \widehat{\pi}^{-1}(T(\beta, \gamma))$, then it follows that $\widehat{\pi}:\widetilde{M} \rightarrow T(\beta, \gamma)$ is surjective. 

Now, consider the case that $\alpha =0, \gamma \in \mathbb{Z}$ and $\beta \neq 1$. The element $\phi(e_k,w_{-k})$ is mapped to $k(\beta -1) w_0 + \tau (k,-k) v_0$. If $\tau(k,-k) \neq 0$ for some $k  \in \mathbb{Z}$, then, as $\pi$ is surjective onto $T(\beta, 0)$, there exists some $\sum_i \phi(x_i, u_i) \in \widetilde{M}$ such that $\pi \left( \sum_i \phi(x_i, u_i ) \right) = w_0$ so that $\pi \left( \frac{1}{\tau(k,-k)} \phi(e_k, w_{-k}) + \sum_i \phi(x_i, u_i) \right) = v_0$.

If $\tau(k,-k) =0$ for every $k \in \mathbb{Z}$, then we can obtain an equivalent cocycle by adding the coboundary $\tau(k,m) = (\beta -1) k v_0$, which is exactly the coboundary given by $-(e_k \varphi)(w_m)$ where $\varphi(w_i) = v_i$. Then this equivalent cocycle is nonzero for $k \neq 0$ so that we may apply the above argument.
\end{proof}

Now, for $\beta \neq 1$, we have the short exact sequence 
\[
 0 \rightarrow \widehat{M}(\alpha,\gamma) \rightarrow \widetilde{M} \rightarrow T(\beta,\gamma) \rightarrow 0
\]
All these modules are cuspidal; this follows as $\widehat{M}(\alpha,\gamma)$ and $ \widetilde{M}$ are $AW_1$-submodules of $ \widehat{M}$. 

The next proposition will be very important in showing that for the generic case, 1-cocycles are polynomial functions. First, we introduce the maps in Hom$(A, T(\alpha, \gamma))$ for $k \in \gamma + \mathbb{Z}, i \in \mathbb{Z}_+$, 
\[
 \theta_k^{(i)} : t^m \rightarrow \dfrac{(m+k)^i}{i!} v_{m+k}, \, \, \, \, \, 
\delta_{k}: t^m \rightarrow 
 \begin{cases}
  v_{m+k}, & m=-k \\
  0, & m \neq -k \\
 \end{cases}
\]

Notice the $A$-action on $\theta_m^{(i)}$ is given by $t^m \theta_k^{(i)} = \theta_{k+m}^{(i)}$.

\begin{remark}
When the quotient module $T(\beta,\gamma)$ is isomorphic to the submodule $T(\alpha,\gamma)$, then the maps $\theta_k^{(0)}=\varepsilon_k$ and $\theta_k^{(1)}=\eta_k$. We make the distinction between these functions for convenience of notation later on. 
\end{remark}

\begin{proposition}\label{Prop} Any $AW_1$-submodule in Hom(A,T($\alpha,\gamma$)) with finite dimensional weight spaces is contained in a submodule spanned by:
\begin{enumerate}
 \item $\left\lbrace \theta_k^{(0)}, \cdots , \theta_k^{(N)} | k \in \gamma + \mathbb{Z} \right\rbrace$ for some $N \in \mathbb{N}$, when $\alpha \neq 0$,
 \item  $\left\lbrace \theta_k^{(0)}, \cdots , \theta_k^{(N)}, \delta_{k} | k \in \gamma + \mathbb{Z} \right\rbrace$ for some $N \in \mathbb{N}$, when $\alpha =0$.
\end{enumerate}
\end{proposition}

\begin{proof}
Suppose that $\varphi \in$ $\text{Hom}(A, T(\alpha,\gamma))$. Without loss of generality, suppose that $\varphi$ is an element of weight $k$ and $\varphi(t^m) = a_m v_{m+k}$. We will show that the function $a(m) = a_m$ is a polynomial in $m$ when $\alpha \neq 0$.

First, we consider $t^{-i}e_i\varphi \in \text{Hom}(A, T(\alpha,\gamma))$. In this notation, $t^{-i}e_i\varphi = t^{-i} (e_i \varphi)$ which is not the same as $(t^{-i} e_i) \varphi$. Since $e_i \varphi$ is an element with weight $k+i$, and $t^{-i}$ is an element of weight $-i$, $t^{-i}e_i\varphi$ is still an element of weight $k$ in $\text{Hom}(A, T(\alpha ,\gamma))$. Then $(t^{-i}e_i\varphi)(t^m) = b_m v_{m+k}=b(m) v_{m+k}$. The functions $a$ and $b$ are connected by the following relation.
\[
 b(m) = m(a(m-i) - a(m)) + (k-i - \alpha i) a(m-i) + i a(m)
\]
Define an action of $t^{-i}e_i$ on $a$ as 
\[
 (t^{-i}e_i a )(m)= m(a(m-i) - a(m)) + (k-i - \alpha i) a(m-i) + i a(m)
\]
and let $ z_n = \sum_{i=0}^{n+1} (-1)^i {n+1 \choose i} t^{i-1}e_{1-i}$ and $ y_n = \sum_{i=0}^{n} (-1)^i {n \choose i} t^i e_{-i}$. One can easily check that the following Lemma holds:

\begin{lemma}\label{z_n}
  $z_n = z_{n-1} - y_n$ for all $n \in \mathbb{N}$.
\end{lemma}

We can give an explicit form of $z_n$ and $y_n$, by
 simplifying using that $\sum_{i=0}^j (-1)^i {j \choose i} =0$ for $j\geq 1$ and $\sum_{i=0}^j (-1)^i {j \choose i} i=0$ for $j \geq 2$:
\begin{align*}
 (z_n a)(m)  = \sum_{i=0}^{n+1}& (-1)^i {n+1 \choose i} (m +k + (\alpha+1)(i-1) ) a(m+i-1) , \text{ for }n \geq 1
 \\
 (y_{n}a)(m) = \sum_{i=0}^{n}& (-1)^i {n \choose i} \left(  m +k+i + \alpha i \right) a(m+i) , \text{ for }n \geq 2
\end{align*}

Observe that $a$ is a function from $\gamma + \mathbb{Z}$ to $\mathbb{C}$. By assumption, $\varphi$ is contained in an $AW_1$-module with finite dimensional weight spaces. Hence $z_n$ in a linear operator on an $\ell$-dimensional space and $z_n \in M_{\ell \times \ell}(\mathbb{C})$.

By Lemma 4 of \ref{jets}, $[z_{-1}, z_n] = -nz_n$ for all $n\in \mathbb{N}$, which means that $z_n$ will be an eigenvector of $ad z_{-1}$. As $z_n \in M_{\ell \times \ell}$, $adz_{-1} \in M_{\ell^2 \times \ell^2} (\mathbb{C})$. Thus $adz_{-1}$ has at most $\ell^2$ unique eigenvalues, so that $z_n$ can only be non-zero for finitely many $n \in \mathbb{N}$. 

Let $\bar{n} = \text{max}\lbrace n \in \mathbb{N} | z_n \neq 0 \rbrace$ if this set is non-empty, and set $\bar{n}=1$ if the set is empty. Then $z_{n}=z_{n+1} =0$ for all $n > \bar{n}$, so $z_n - z_{n+1}$ is zero for all $n > \bar{n}$. It follows from the previous lemma that $y_{n+1}$ is zero for $n > \bar{n}$. We observe
\begin{equation}
\begin{aligned}
  (y_{n+1}a)(m-1) =  \sum_{i=0}^{n+1} (-1)^i {n+1 \choose i}& \left(  m +k +(\alpha +1)i -1 \right) a(m+i-1) 
  \end{aligned}
\end{equation}
and so by shifting the variable $m$ to $m-1$, $y_{n+1}$ is still the zero function. So
\begin{align*}
  (0 \cdot a)(m) = &((z_n -y_{n+1})\cdot a)(m) \\
 =&-\sum_{i=0}^{n+1} (-1)^i {n+1 \choose i} \alpha a(m+i-1)  
 \end{align*}
As long as $\alpha \neq 0$, 
\begin{equation}
 0=\sum_{i=0}^{n+1} (-1)^i {n+1 \choose i} a(m+i-1)  
\end{equation}
for all $n > \bar{n}$ and hence $\lbrace a_m \rbrace$ satisfy recurrence relations. This in turn tells us that each $a$ is a polynomial in $m$ by the use of Lemma 3 in \ref{jets}, so that $a(m)$ is in the span $(m+k)^i $ for $i=1, \dots, N$, where  where $N= \text{deg}(a)$. Thus $\varphi$ is in the span of $\left\lbrace \theta_m^{(0)}, \dots , \theta_m^{(N)} \right\rbrace$.

In the case that $\alpha=0$,
\[
 (z_n a)(m) =\sum_{i=0}^{n+1} (-1)^i {n+1 \choose i} ( m +i-1+k)a(m+i-1)
\]
Set $d(m)= (m+k) a(m)$ so that
\[
 (z_n a)(m) =\sum_{i=0}^{n+1} (-1)^i {n+1 \choose i}d(m+i-1)=0
\]
By the previous argument, $d$ is a polynomial in $m$ with a root at $-k$. So there exists a polynomial $g(m)$ such that $d(m) = (m+k)g(m)$. Thus, $a(m)$ and $g(m)$ agree on every integer except $-k$ so that
\begin{equation}
a(m) = g(m) + c \delta_{m,-k}
\end{equation}
for some $c \in \mathbb{C}$, where $g(m) = \frac{d(m)}{m+k}$ is a polynomial in $m$.

This suggests that $a$ may have a delta-function component. As 
\[
t^{\ell} \delta_k = \delta_{k+\ell},\hspace{1cm} e_\ell \delta_k = (k+\ell) \delta_{k+\ell}
\]
this indeed contained in an $AW_1$-submodule with finite dimensional weight spaces. 
\end{proof}

\subsection{Finding an appropriate basis for $M$}

For the rest of this section, we will assume that $\beta \neq 1$. From here, we would like to obtain a basis of $M$ that will admit polynomial cocycles. In particular, we would like to make use of the last proposition to show that we can obtain a basis of $M$ so that all possible 1-cocycles are contained in subspaces of the form described in the last proposition. 

Let $\mathscr{L}$ be the Lie algebra spanned by the elements $z^i \frac{d}{dz}$ where $i \in \mathbb{Z}$. Denote $\mathscr{L}_+$ to be the subalgebra spanned by the elements $z^i \frac{d}{dz}$ where $i \in \mathbb{Z}$ and $i \geq 1$.

\begin{theorem}[Theorem 4.11 of \ref{W_nmodules}]\label{theorem}
If $V$ is a cuspidal $AW_1$-module with weights in $\gamma + \mathbb{Z}$ for some $\gamma \in \mathbb{C}$, then there exists a finite dimensional module $(U, \rho)$ of $\mathscr{L}_+$ such that 
\[
V \cong A \otimes U
\]
and with $W_1$-action given by
\[
e_k (t^m \otimes u) = (m+\gamma)t^{m+k} \otimes u + \sum_{i=1}^\infty \dfrac{k^i}{i!}t^{m+k} \otimes \rho \left( z^i \frac{d}{dz} \right) u
\]
for all $k,m \in \mathbb{Z}, u \in U$.
\end{theorem}

A quick application of this theorem is to look at the module $\widehat{T(\beta, \gamma)}$. As this space is spanned by $\varepsilon_i$ and $\eta_i$ for $i \in \gamma +\mathbb{Z}$, then if $\varepsilon_j = t^{j-\gamma} \otimes \varepsilon$, $\eta_j = t^{j-\gamma} \otimes \eta$, so $\widehat{T(\beta,\gamma)} \cong A \otimes U'$, where $ U' = \langle \varepsilon, \eta \rangle \subset \widehat{T(\beta,\gamma)}_\gamma$.

The $W_1$-action on these elements is given by:
\begin{equation}\label{calculations}
e_k \varepsilon_j = (j + \beta k) \varepsilon_{j+k} 
\end{equation}
\begin{equation}\label{cal2}
e_k \eta_j = (j + k (\beta -1)) \eta_{j+k} - k^2 (\beta -1) \varepsilon_{j+k}
\end{equation}

By the previous theorem, we can derive the representation $\rho$: 

\begin{align*}
 \rho \left( z \dfrac{d}{dz} \right) &= \begin{pmatrix} \beta & 0 \\ 0 & \beta -1 \end{pmatrix}, &
\rho \left( z^2 \dfrac{d}{dz} \right)  &= \begin{pmatrix} 0& -2(\beta-1) \\ 0&0 \end{pmatrix}, & \rho \left( z^i \dfrac{d}{dz} \right)  &=0, \; \forall i \geq 3.
\end{align*}

Thus, $\rho \left( z \frac{d}{dz} \right)$ is the only element that acts non-trivially on $\varepsilon$, so that $e_k(t^m \otimes \varepsilon) = (m+\gamma)t^{m+k} \otimes \varepsilon + k t^{m+k}\otimes \beta  \varepsilon$.

Let us make use of the theorem for $\widetilde{M}$. Define $\sigma_0 \in \widehat{M}$ such that 
\begin{equation}\label{sigdef}
  \sigma_0 = \dfrac{1}{1-\beta}\left( \phi(e_0, w_\gamma) - \phi(e_1, w_{\gamma -1} ) \right) 
\end{equation}  
where $w_{\gamma -1},w_\gamma \in M$, the preimages of the basis vectors $w_{\gamma -1},w_\gamma \in T(\beta,\gamma)$. Then $\widehat{\pi} (\sigma_0)= \varepsilon_\gamma$
so $\sigma_0 \in \widehat{\pi}^{-1} (T(\beta, \gamma)) = \widetilde{M}$.

We make use of this element to define a basis of $M$. Let $\sigma_m = t^m \sigma_0$ and define $\widehat{w}_m = \pi(\sigma_m) \in M$. It follows that $ \widehat{w}_m 
= \sigma_0(t^m)$.

As we have already concluded, $\widetilde{M}$ is a submodule of a cuspidal $AW_1$-module, and thus is itself a cuspidal $AW_1$-module. 
By Theorem \ref{theorem}, we can write $\widetilde{M} = A \otimes \widetilde{M}_\gamma$, and the $\gamma$-weight space $\widetilde{M}_\gamma$ admits the action of $\mathscr{L}_+$. Notice $\sigma_0 \in \widetilde{M}_\gamma$ so that $\sigma_0 = 1 \otimes \sigma$ for some $\sigma \in \widetilde{M}_\gamma$.
Then
\begin{equation}\label{sigmanaught}
\begin{aligned}
 e_k\sigma_0 =& e_k (1 \otimes \sigma) =& \gamma t^k \otimes \sigma + t^k \otimes ( k u_1 + k^2 u_2 + \cdots + k^n u_n)
 \end{aligned}
\end{equation}
for all $k \in \mathbb{Z}$ where $u_i = \frac{1}{i!} \rho \left( z^i \frac{d}{dz} \right) \sigma$. As $\sigma \in \widetilde{M}_\gamma$, each $u_i \in \widetilde{M}_\gamma$. Then $e_k \sigma_0 \in \widetilde{M}_{k+\gamma}$. We can use this action to describe the $W_1$-action on the basis vectors $\widehat{w}_m$.
\begin{align*}
 e_k \widehat{w}_m 
&= (e_k \sigma_0)(t^m) + \sigma_0(e_k t^m) \\
&= (1 \otimes (\gamma \sigma +k u_1 + k^2 u_2 + \cdots + k^n u_n))(t^{m+k}) + m \sigma_0(t^{k+m}) \\
&= (m + \gamma + \beta k) \widehat{w}_{k+m} + \zeta(k,m) \widehat{w}_{k+m} +\tau(k,m) v_{k+m}
\end{align*}

The goal is to show that $\zeta$ is the zero function and $\tau(k,m)$ is a polynomial function in $k$ and $m$.
To do this, we will need to make use of both Proposition \ref{Prop} and equation (\ref{sigmanaught}). The problem is that the first result is in the context of Hom$(A, T(\alpha,\gamma))$ while the second result is in the context of $\mathscr{L}_+$-modules. The next lemma will give us a way to relate $AW_1$-modules to $\mathscr{L}_+$-modules.

\begin{lemma}\label{primemap}
Suppose $\varphi: A \otimes X \rightarrow A \otimes Y$ is a homomorphism of $AW_1$-modules for some $\mathscr{L}_+$-modules $X,Y$ and $1 \otimes X$ and $1 \otimes Y$ both have weight $\gamma$. Then the restriction map $\varphi': 1 \otimes X \rightarrow 1 \otimes Y$ is a homomorphism of $\mathscr{L}_+$-modules.
\end{lemma}

\begin{proof}
To simplify notations, let $u_i = \frac{1}{i!}\rho ( z^i \frac{d}{dz})$. 
First notice that since $\varphi$ is an $A$-module homomorphism and a $W_1$-module homomorphism, 
$ 
 \varphi(t^m \otimes x) = t^m \cdot \varphi(1 \otimes x) = t^m \otimes \varphi'(x)
$ and
\begin{equation}\label{varphi1}
 \varphi(e_k (t^m \otimes x)) = e_k (\varphi(t^m \otimes x))
\end{equation}

By making us of Theorem \ref{theorem}, equation (\ref{varphi1}) becomes
\begin{align*}
\varphi \left( t^{m+k} \otimes  \left((m+\gamma) x + \sum_{i=0}^{\infty} k^i u_i x \right) \right) =& e_k (t^m \otimes \varphi'(x)) \\
 t^{m+k} \otimes \varphi ' \left( (m+\gamma) x + \sum_{i=0}^{\infty} k^i u_i x \right) =& t^{m+k} \otimes \left( (m+\gamma) \varphi'(x) + \sum_{i=0}^{\infty}k^i u_i \varphi'(x) \right) 
\end{align*}

Then the following is true for all $k \in \mathbb{Z}$:
\begin{equation}
 \varphi '\left(   \sum_{i=1}^{\infty} \frac{k^i}{i!}\rho \left( z^i \frac{d}{dz}\right) x \right)  = \left(  \sum_{i=1}^{\infty} \frac{k^i}{i!}\rho \left( z^i \frac{d}{dz} \right) \varphi'(x) \right)
\end{equation}

 These will be polynomial functions since for all $ x \in \mathscr{L}_+$, there exists $n \in \mathbb{N}$ such that $\rho\left( z^n \frac{d}{dz} \right)x=0$ by Lemma 2 of \ref{jets}. Since these functions are equal on all integer values then necessarily they must be equal as polynomial functions. Therefore,
\[
\varphi '\left( \rho \left( z^i\frac{d}{dz} \right) x \right)  = \rho \left( z^i \frac{d}{dz} \right) \varphi'(x),  \forall i \geq 1, \forall x \in X
\]
 Thus, $\varphi'$ is a homomorphism of $\mathscr{L}_+$-modules.
\end{proof}

Now, by defining $\widehat{\pi}' : 1 \otimes \widetilde{M}_\gamma \rightarrow 1 \otimes \widehat{T(\beta, \gamma)}_\gamma$, by the above results, $\widehat{\pi}'$ is a homomorphism of $\mathscr{L}_+$-modules that is surjective and 
its kernel is such that $\widehat{M}(\alpha, \gamma) \cong A \otimes \text{ker}(\widehat{\pi}')$.

This map $\widehat{\pi}'$ will prove useful in showing the funtion $\tau(k,m)$ is a polynomial in $k$ and $m$.
Since $1 \otimes \widehat{\pi}'(\sigma)  = \widehat{\pi}(\sigma_0) = \varepsilon_\gamma = 1 \otimes \varepsilon$, we see that $\widehat{\pi}'(\sigma) = \varepsilon$ so that
\begin{align*}
\widehat{\pi}(e_k \sigma_0) &= \gamma t^k \otimes \widehat{\pi}'(\sigma) +t^k \otimes \widehat{\pi}'\left( \sum_{i=1}^{\infty} \frac{k^i}{i!} \rho \left( z^i \dfrac{d}{dz} \right) \sigma \right) \\
&= \gamma t^k \otimes \varepsilon + t^k \otimes \left( \sum_{i=1}^{\infty} \frac{k^i}{i!} \rho \left( z^i \dfrac{d}{dz} \right) \varepsilon  \right) \\
&= t^k \otimes \left( \gamma + \beta k \right) \varepsilon
\end{align*}

By applying $\widehat{\pi}$ to (\ref{sigmanaught}),
\[
\gamma t^k \otimes \widehat{\pi}' (\sigma) + t^k \otimes (k\widehat{\pi}'( u_1) + k^2 \widehat{\pi}'(u_2) + \cdots + k^n \widehat{\pi}'( u_n)) =  t^k \otimes (\gamma +\beta k ) \varepsilon
\]
so that  $k^2 \widehat{\pi}'(u_2) + \cdots + k^n \widehat{\pi}'(u_n) =0$ for all $k \in \mathbb{Z}$. Since $\lbrace k^2, k^3, \dots , k^n \rbrace$ are linearly independent, this implies that $\widehat{\pi}'(u_i) =0$ for $2 \leq i \leq n$. But then $\lbrace u_i |2 \leq i \leq n  \rbrace \subset \text{ker}(\widehat{\pi}')$ so that $t^k \otimes u_i \in \widehat{M}(\alpha,\gamma)$ for $2 \leq i \leq n$. 

We now have that $ t^k \otimes k (\beta \varepsilon) = t^k \otimes (k \widehat{\pi}' (u_1)) $ so that $
 t^k \otimes 0 = t^k \otimes k ( \widehat{\pi}'(u_i) - \beta \varepsilon)$.
Thus $t^k \otimes k (\widehat{\pi}'(u_1) - \beta \varepsilon)= t^k \otimes k \widehat{\pi}'(\omega) =0$ for some $\omega \in \text{ker}(\widehat{\pi}')=\widehat{M}(\alpha, \gamma)$. We conclude that $\widehat{\pi}'(u_1) = \beta \varepsilon  +\widehat{\pi}'(\omega)$. 
so that we may simplify (\ref{sigmanaught}) to:
\[
e_k \sigma_0 =(\gamma +\beta k)t^k \otimes \sigma + t^k \otimes(k \omega + k^2 u_2 + \cdots + k^n u_n )
\]

Since $u_i \in \widehat{M}(\alpha,\gamma)$ and $\omega \in \widehat{M}(\alpha, \gamma)_\gamma$, then we may apply Proposition \ref{Prop} to write $u_i = \sum_{j=0}^{N_i} c_{i,j} \theta_\gamma^{(j)} + c_{i, N_i +1}\delta_\gamma$ for $2 \leq i \leq n$, $N_i \in \mathbb{N}$ and $c_{i,j} \in \mathbb{C}$. Similarly, $\omega = \sum_{j=0}^{N_1} c_{1, j} \theta^{(j)}_\gamma + c_{1, N_1 +1} \delta_\gamma$ for $N_1 \in \mathbb{N}$ and $ c_{1,j} \in \mathbb{C}$. Then $t^k \otimes (k\omega+ k^2u_2 + \cdots + k^n  u_n) = t^k \otimes \left( \sum_{i=1}^n k^i \left( \sum_{j=0}^{N_i} c_{i,j} \theta_\gamma^{(j)} + c_{i, N_i +1} \delta_\gamma \right)\right)$.

Finally, we can say something about the $W_1$-action on $\widehat{w}_m$.
\begin{align*}
 {e_k \widehat{w}_m} 
&= {e_k \pi(\sigma_m)} \\
&= e_k \sigma_0 (t^m)\\
&= \gamma \widehat{w}_{m+k}+t^k \otimes (k u_1 + k^2 u_2 + \cdots + k^n u_n)(t^m)+m \widehat{w}_{k+m}\\
&= t^k \otimes k \beta \sigma (t^m) + t^k \otimes \left( \sum_{i=1}^n k^i \left( \sum_{j=0}^{N_i} c_{i,j} \theta_\gamma^{(j)} + c_{i, N_i +1} \delta_\gamma \right)\right)(t^{m}) + (m +\gamma){\widehat{w}_{k+m}}\\
&= (m+\gamma+\beta k) \widehat{w}_{k+m} + \left( \sum_{i=1}^n k^i \left( \sum_{j=0}^{N_i} c_{i,j} \theta_\gamma^{(j)} + c_{i, N_i +1} \delta_\gamma \right)\right)(t^{m+k})
\end{align*}

Notice that since for all $k,m \in \mathbb{Z}$,
\begin{align*}
\sum_{i=1}^n k^i \left( \sum_{j=0}^{N_i} c_{i,j} \theta_\gamma^{(j)} + c_{i, N_i +1} \delta_\gamma \right)(t^{m+k}) &= t^k \otimes  \sum_{i=1}^n k^i \left( \sum_{j=0}^{N_i} c_{i,j} \theta_\gamma^{(j)} + c_{i, N_i +1} \delta_\gamma \right)(t^m) \\
&= \pi \left(t^{m+k} \otimes\sum_{i=1}^n k^i \left( \sum_{j=0}^{N_i} c_{i,j} \theta_\gamma^{(j)} + c_{i, N_i +1} \delta_\gamma \right)\right) \\
&= \pi \left( 1\otimes \sum_{i=1}^n k^i \left( \sum_{j=0}^{N_i} c_{i,j} \theta_{m+k+\gamma}^{(j)} + c_{i, N_i +1} \delta_{m+k+\gamma} \right)\right)
\end{align*}
We may think of $t^{m+k} \otimes\sum_{i=1}^n k^i \left( \sum_{j=0}^{N_i} c_{i,j} \theta_\gamma^{(j)} + c_{i, N_i +1} \delta_\gamma \right)$ as the image under $\pi$ of a polynomial $\widetilde{\tau}(k,m)$ in $k$ and $\theta_{m+k+\gamma}$ functions with a possible $\delta_{k+m+\gamma}$ function. Thus,
\[
\pi( \widetilde{\tau}(k,m)) = \tau(k,m) v_{m+k+\gamma}
\]
where $\tau(k,m)$ is a polynomial in $k$ and $m$, with a possible delta function. We may shift $m\in \mathbb{Z}$ to $m \in \gamma + \mathbb{Z}$ to simplify our notation. 

Now we apply the surjective map $\overline{\cdot} :M \rightarrow T(\beta,\gamma)$.
Since $\widetilde{\tau}(k,m) \in  \widehat{M}(\alpha,\gamma)$, using  $\pi(\widehat{M}(\alpha,\gamma)) \subset T(\alpha,\gamma)$ we obtain
\begin{align*}
 \overline{e_k \widehat{w}_m} &=  \overline{(m+\beta k) \widehat{w}_{m+k}}+\overline{\pi(\widetilde{\tau}(k,m))} \\
 &= (m+ \beta k) w_{k+m}
\end{align*}

This shows that $\zeta(k,m) = 0$, so that $e_k \widehat{w}_m = (m+ \beta k) \widehat{w}_{k+m} + \tau(k,m) v_{k+m}$. We can identify $T(\beta, \gamma)$ in $\widetilde{M}$ by using basis vectors $\varepsilon_m=t^m \otimes \sigma$ so that $\varepsilon_m(1) = \widehat{w}_m$. The $W_1$-action on $\widetilde{M}$ is given by
\[
e_k \varepsilon_m = (m+\beta k)\varepsilon_{k+m} + \widetilde{\tau}(k,m)
\]
where $\widetilde{\tau}(k,m)$ is contained in a submodule of the form in Proposition \ref{Prop}. 

This is enough to show that our cocycles are combinations of polynomials and delta functions. By Proposition \ref{Prop}, we know that $\widetilde{\tau}(k,m)$ is contained in the submodule $\lbrace \theta_k^{(0)}, \dots, \theta_k^{(N)}, \delta_n \rbrace$ for some $N$. Thus, $\widehat{M}(\alpha,\gamma)$ is contained in a module $\widehat{M}_N(\alpha,\gamma)$ where $\widehat{M}_N(\alpha,\gamma)$ is the submodule spanned by $\lbrace \theta_k^{(0)}, \dots, \theta_k^{(N)}, \delta_k \rbrace$. As $\widehat{M}(\alpha,\gamma) \subseteq \widehat{M}_N(\alpha,\gamma)$, we may extend our module $\widetilde{M}$ to a module $\widetilde{M}_N$ given by the extension
\[
0 \rightarrow \widehat{M}_N(\alpha,\gamma) \rightarrow \widetilde{M}_N \rightarrow T(\beta, \gamma) \rightarrow 0
\]
This will be convenient in the next section when we consider $\widetilde{\tau}$ cocycles as we may assume that $\theta_k^{(i)}$ is contained in our submodule for each $i \in \mathbb{N}$. As $\widetilde{M} \subset \widetilde{M}_N$, then $\pi: \widetilde{M}_N \rightarrow M$ will be surjective for $\beta \neq 1$, which will allow us reclaim the cocycles from the original extension of $M$. 

\section{Polynomial Cocycles}

As long as $\alpha \neq 0$ and $\beta \neq 1$, every cocycle is a polynomial function. For this section, we consider only polynomial functions, even possibily in these two cases. As mentioned in the introduction, we will do these calculations with $\widetilde{\tau}(k,m)$ in the context of $\widetilde{M}_N$ rather than $M$; the homomorphism $\pi$ can be used to recover the functions $\tau(k,m)$. 

\subsection{General results for polynomial cocycles}

The $W_1$-action on $\varepsilon_m$ is given by
\[
 e_k \varepsilon_m = (m+ \beta k) \varepsilon_{k+m} + \sum_{i=0}^{n} \sum_{\ell=1}^{s} c_i k^{\ell} \theta_{k+m}^{(i)}
\]
where $k \in \mathbb{Z}, n,s \in \mathbb{Z}_+$ and $m \in \gamma + \mathbb{Z}$. The space of polynomials $\mathbb{C}[k,m]$ admits a $\mathbb{Z}$-grading, where homogeneous elements of degree $n$ consist of monomials whose powers of $k$ and $m$ sum to $n$. Using (\ref{taucondition}), we can see that each homogeneous component of the polynomial $\tau(k,m)$ must independently satisfy the cocycle condition. Similarly, this is true for the coboundary condition. 

For the cocycle $\widetilde{\tau}(k,m)$, we may introduce an analogous idea of a homogeneous element. A homogeneous element of degree $n$ will consist of monomials $k^\ell \theta_{k+m}^{(i)}$ for which $\ell +i =n$. In this way, we obtain a $\mathbb{Z}$-grading on the 
cohomology space. Again, each homogeneous component of $\widetilde{\tau}(k,m)$ will independently satisfy (\ref{taucondition}) and (\ref{trivial}), so that it is enough to consider homogeneous cocycles:
\begin{equation}\label{epsilon}
 e_k \varepsilon_m = (m+\beta k)\varepsilon_{m+k} + \sum_{i=0}^{n-1} c_i k^{n-i} \theta_{m+k}^{(i)}
\end{equation}

By Theorem \ref{theorem},
\begin{equation}\label{rhoepsilon}
e_k \varepsilon_m = m \varepsilon_{m+k} + \sum_{i=1}^{\infty}\dfrac{k^i}{i!} \left( \rho \left(z^i \dfrac{d}{dz}\right) \varepsilon \right) _{(m+k)}
\end{equation}
\begin{equation}\label{eq3}
e_k \theta_m^{(p)} = m \theta_{m+k}^{(p)} + \sum_{i=1}^{\infty}\dfrac{k^i}{i!} \left( \rho \left( z^i \dfrac{d}{dz} \right) \theta^{(p)} \right)_{(m+k)}
\end{equation}

\begin{lemma}
The $W_1$-action on $\theta_j^{(p)}$ is given by
\begin{equation}\label{theta}
 \begin{aligned}
 e_k \theta_j^{(p)} =& \sum_{i=0}^{p-1} \dfrac{(-1)^{p-i}}{(p+1-i)!} \left( (p+1-i) \alpha -(p+1)\right) k^{p+1-i} \theta_{j+k}^{(i)}
 \\ & + (j+(\alpha-p)k)\theta_{j+k}^{(p)} 
 \end{aligned}
\end{equation}
\end{lemma}

This is directly verified by the $W_1$-action defined in Section 1 so the proof is omitted.

Here, we are interested in finding all non-trivial polynomial cocycles. It is crucial that we know the form of the trivial homogeneous polynomial coboundaries. For $\widetilde{\tau}(k,m)$, a 1-coboundary is of the form $-(e_k \varphi)(\varepsilon_i)$ for $\varphi \in \text{Hom}(T(\beta, \gamma), \widehat{M}_N(\alpha, \gamma))$. Here, $\varphi(\varepsilon_i) = g(i) \theta_i^{(n)}$, where $g: \gamma + \mathbb{Z} \rightarrow \mathbb{C}$. Then coboundaries are of the form
\begin{equation}\label{coboundaries}
 \begin{aligned}
\widetilde{\tau}(k,m) =& \, m (g(m+k) - g(m)) \theta_{m+k}^{(n)} + (\beta k g(m+k) -(\alpha-n)k g(m)) \theta_{m+k}^{(n)} \\
 &- \sum_{i=0}^{n-1} \dfrac{(-1)^{n-i}}{(n+1-i)!} \left( (n+1-i) \alpha -(n+1)\right) g(m) k^{n+1-i} \theta_{m+k}^{(i)}
 \end{aligned}
\end{equation}

In general, the map $g$ has very few restrictions. In this case though, we would like to find coboundaries that are homogeneous polynomials in $k$ and $\theta_{m+k}^{(i)}$. The following lemma outlines the restrictions on $g$ that make coboundaries of interest to us.

\begin{lemma}\label{g(m)=delta}
The function $g(x)$ makes (\ref{coboundaries}) a homogeneous polynomial 1-coboundary in $k$ and $\theta_{m+k}^{(i)}$ only when:
\begin{enumerate}
 \item $g(x)$ is a constant function; this admits coboundaries that are homogeneous polynomials in $k$ and $\theta$.
 \item $g(x)$ is a rational function, $\alpha=0$ and $\beta =1$; this admits the delta-function coboundary $ \delta_{m,0} - \delta_{m+k,0}$.
\end{enumerate}
\end{lemma}

\begin{proof}
First, we assume that $\widetilde{\tau}(k,m)$ is of the form described in Proposition \ref{Prop}. By comparing both sides of (\ref{coboundaries}) for $\alpha \neq 1$, the coefficient of $\theta^{(0)}_{m+k}$ must be a polynomial in $k$ that is independent of $m$. Thus for $n\geq 1$, $g$ must be a constant function. If $\alpha=1$ and $n\geq 2$, we may consider the coefficient of $\theta_{m+k}^{(1)}$ to come to the same conclusion. 

We need to consider the cases of $n=0$ and ($n=1$ and $\alpha=1$). The coefficient of $\theta^{(n)}$ is given by
\begin{equation}\label{coefficient}
(m+\beta k) g(m+k) - (m+ (\alpha -n) k) g(m)
\end{equation}
This must be a polynomial in $k$. If we set $m = \gamma$ in (\ref{coefficient}) 
then as $g(\gamma)$ is a constant, $(\gamma + \beta k) g(\gamma +k)$ must be a polynomial in $k$. So $g(x) = \frac{q_1(x)}{(\beta x - (1-\beta)\gamma)}$ for some polynomial $q_1(x)$. Similarly, if we set $m=\gamma+1$, 
we conclude that $g(x) = \frac{q_2(x)}{(\beta x + (1-\beta)(\gamma+1))}$ for some polynomial $q_2(x)$. 
Thus, 
\[
\frac{q_1(x)}{(\beta x - (1-\beta)\gamma)} =  \frac{q_2(x)}{(\beta x + (1-\beta)(\gamma+1))}
\]
so that either $g$ is a polynomial or $(\beta x - (1-\beta)\gamma) = (\beta x + (1-\beta)(\gamma+1))$. 
If these two factors are equal, then $\beta =1$ and for some polynomial $q(x)$ and some constant $c$, 
\[
g(x) = \left\lbrace \begin{array}{ll}
\dfrac{q(x)}{x}, & x \neq 0\\
c, & x=0 
\end{array}
\right.
\]
The coefficient (\ref{coefficient}) at $m=\gamma$ becomes
$
q(\gamma+k) - \frac{(\gamma+(\alpha-n)k)}{\gamma}q(\gamma)
$
and thus $\alpha=n$ to obtain a polynomial coboundary. Notice for $n=1, \alpha=1$ this is already true while for $n=0$, this requires that $\alpha=0$. This reduces (\ref{coefficient}) to
\[
(m+k)g(m+k) - m g(m) = q(m+k) - q(m) + \delta_{m,0} q(m) - \delta_{m+k, 0}q(m+k)
\]
for ($\alpha=0, n=0$) and ($\alpha=1, n=1$).
\end{proof}

Therefore, for strictly homogeneous polynomial coboundaries of degree $n$, $\widetilde{\tau} \sim 0$ if there exists $ h \in 
\mathbb{C}$ such that 
\begin{equation}\label{trivial}
 \begin{aligned}
\widetilde{\tau}(k,m) =&h \sum_{i=0}^{n-2} \dfrac{(-1)^{n-1-i}}{(n-i)!}  \left((n-i) \alpha - (n) \right) k^{n-i} \theta_{m+k}^{(i)}
  \\  
  & + h k (\alpha - \beta - n +1) \theta_{m+k}^{(n-1)}
 \end{aligned}
\end{equation}

We can obtain some more information about the representation $\rho$, which will be a useful our calculations. From (\ref{epsilon}) - (\ref{theta}), we can derive the formulae:
\begin{align}
 \rho\left(z \dfrac{d}{dz}\right)\theta^{(p)} &= (\alpha -p) \theta^{(p)}, \\
 \rho\left(z^i \dfrac{d}{dz}\right)\theta^{(p)} &=  (-1)^{i-1} \left(i \alpha - (p+1)\right)  \theta^{(p+1-i)}, i \geq 2\\
  \rho\left(z \dfrac{d}{dz}\right)\varepsilon &= \beta\varepsilon + c_{n-1}\theta^{(n-1)}, \\
   \rho\left(z^i \dfrac{d}{dz}\right)\varepsilon &= i! c_{n-i} \theta^{(n-i)}, i \geq 2
\end{align}

Also, since $\rho$ is a representation,
\begin{equation}\label{rhorep}
\left[\rho\left(z^i \dfrac{d}{dz}\right) , \rho\left(z^j \dfrac{d}{dz}\right)\right] u  = (j-i) \rho\left(z^{i+j-1} \dfrac{d}{dz}\right) u
\end{equation}

The next two results will determine the conditions on the coefficients $c_i$ using (\ref{rhorep}).

\begin{lemma}\label{AB} For all cocycles of degree $n$ for $n \in \mathbb{N}_+$, and for $i,j \in \mathbb{Z}_+$,
 \begin{enumerate}
  \item For $i\geq 2$, 
   \begin{equation}\label{B1}
  (\alpha - \beta -(n-1)) c_{n-i} = (-1)^{i-1}(i \alpha -n)   c_{n-1} 
   \end{equation}
  \item For $i,j \geq 2$, $i+j \leq n+1$, and $i \neq j$
   \begin{equation}\label{A1}
   \begin{aligned}
   (i+j-1)!(j-i) c_{n-i-j+1} =& j!(-1)^{i-1} (i \alpha - (n-j+1)) c_{n-j} \\ &- i! (-1)^{j-1} (j \alpha - (n-i+1))c_{n-i} 
   \end{aligned}
   \end{equation}
 \end{enumerate}
\end{lemma}

\begin{proof}
For part (1), set $j=1$, $u=\varepsilon$ so by (\ref{rhorep}) the claim follows immediately. For part (2), by setting $i,j \geq2$ in (\ref{rhorep}), then again the claim follows immediately.
\end{proof}

\begin{theorem}\label{coefficients}
For non-trivial 1-cocycles with $n\geq3$, $c_{n-1} =0$ and $\alpha - \beta =n-1$.
\end{theorem}

\begin{proof}
 Suppose $n \geq 3$ and suppose $c_{n-1}$ is non-zero. Then by (\ref{B1}), $ i!(\alpha-\beta -(n-1)) c_{n-i}$ is non-zero for all $i \neq \frac{n}{\alpha}$. As $n \geq 3$ there exits at least one $1 \leq i \leq n$ such that $ i \neq \frac{n}{\alpha}$; it follows that $(\alpha - \beta -(n-1))$ is necessarily non-zero. Thus
\[
  c_{n-i} = \frac{ (-1)^{i-1}(i \alpha-n) }{i! (\alpha - \beta -(n-1))} c_{n-1}
\]

The $W_1$ action on $\varepsilon_m$ is
\begin{align*}
 e_k \varepsilon_m 
 =& (m+\beta k)\varepsilon_{m+k}  + k c_{n-1} \theta_m^{(n-1)} + \sum_{i=2}^{n} c_{n-i} k^i \theta_m^{(n-i)} 
 \\
 =& (m+\beta k)\varepsilon_{m+k}  + k c_{n-1} \theta_m^{(n-1)} 
 \\
 &+ \frac{c_{n-1}}{(\alpha-\beta -(n-1))}\sum_{i=0}^{n-2} \frac{ (-1)^{n-i-1}((n-i) \alpha -n) }{(n-i)! } k^{n-i} \theta_m^{(i)} 
\end{align*}

Take $h = \frac{c_{n-1}}{\alpha - \beta - (n-1)}\in \mathbb{C}$. Then this 1-cocycle is equivalent to the trivial 1-cocycle. For non-trivial cocycles, $c_{n-1}$ must be zero, so $\alpha-\beta-(n-1)$ must be zero as well.
\end{proof}

The general technique to find cocycles of a specific degree is to first simplify the cocyle using Lemma \ref{AB} and Theorem \ref{coefficients}. We derive the representation $\rho$ using Theorem \ref{theorem} and then make use of (\ref{rhorep}) to derive a system of equations in the coefficients $c_i$. We solve the system of coefficients and find all general solutions, then find all non-trivial cocycles by looking at the solution space distinct from the trivial solution space.

\subsection{Cases when $n \leq 4$}

When $n \leq 4$, there does not exist positive integers $i$ and $j$ satisfying $i \geq 2$, $j \geq 2$ and $i \neq j$ such that $n-i-j+1$ is still positive so that Part (2) of Lemma \ref{AB} cannot be used. Instead, we look at each case individually.

\subsubsection{Cocycles of degree 1}

When $n=1$, we cannot use Part (1) of Lemma \ref{AB}
. Equality (\ref{epsilon}) does give us the possible form of the $W_1$ action:
\[
 e_k \varepsilon_m = (m+\beta k) \varepsilon_{m+k} + c_0 k \theta_{m+k}^{(0)}
\]
From this equation, we can derive the representation $\rho$:
\begin{align*}
 \rho \left (z \dfrac{d}{dz} \right) & = 
\begin{bmatrix}
 \alpha & c_0  \\
 0 & \beta \\
\end{bmatrix}, &
 \rho \left (z^i \dfrac{d}{dz} \right) & = 0, \text{ for } i \geq 2
\end{align*}
For valid cocycles, $\rho$ must satisfy the condition (\ref{rhorep}) but as $\left[ \rho \left( z \frac{d}{dz} \right), \rho \left( z \frac{d}{dz} \right) \right] =0$, this tells us that there are no restrictions on $c_0$ as any value will give a valid representation.

Coboundaries of degree 1 are of the form $hk (\alpha - \beta) \theta_{m+k}^{(0)}$ for some $h \in \mathbb{C}$. If $\alpha \neq \beta$, then $h = \frac{c_0}{\alpha-\beta}$ makes the cocycle trivial. When $\alpha = \beta$, the equivalence relationship is simply zero, so only cocycles strictly equal can be equivalent. Thus for $\alpha = \beta$, we obtain the equivalence class of $k \theta^{(0)}$.

\subsubsection{Cocycles of degree 2}

When $n=2$, Part (1) of Lemma \ref{AB} applies and we obtain the single equation for $i=2$:
\[
 2 (\alpha-\beta -1) c_{0} = -(2 \alpha-2)   c_{1} 
\]
By the proof of Theorem \ref{coefficients}, if $\alpha - \beta -1 \neq 0$, this cocycle will be trivial. Thus, $\alpha - \beta =1$ and we get two cases when the above equation is equal to zero: $c_{1} =0$, or $\alpha =1$.

\hfill

\textbf{Case of $c_{1}=0$}

When $c_1 =0$, $ e_k \varepsilon_m = (m+\beta k)\varepsilon_{m+k} + c_0 k^2 \theta_{m+k}^{(0)}$. We can derive $\rho$ and the system of coefficients we obtain from it gives no added restrictions on $c_0$.

The trivial cocycle is $-(\alpha -1) k^2 \theta_{m+k}^{(0)}$
so if $\alpha \neq 1$, every cocycle is trivial by setting $h = \frac{c_0}{\alpha -1}$. Therefore, from this case we obtain an equivalence class of $ k^2 \theta^{(0)} \text{ for } \alpha =1$ and $ \beta =0$.

\hfill

\textbf{Case of $\alpha =1$}

When $\alpha =1$, then $\beta =0$ and $ e_k \varepsilon_m= m \varepsilon_{m+k} + c_0 k^2 \theta_{m+k}^{(0)} + c_1 k \theta_{m+k}^{(1)}$.
Again, this gives 
no restrictions on $c_0$ or $c_1$. The coboundary of degree 2 is exactly zero when $\alpha =1$. This gives us the equivalence class of $k \theta^{(1)} \text{ for } \alpha =1$ and $ \beta =0$.

\subsubsection{Cocycles of degree 3}

For the case of $n=3$, we can now make use of Theorem \ref{coefficients}. For non-trivial cocycles, $\alpha - \beta = 2$, and $c_{2} =0$. Then, by (\ref{epsilon}),
\[
 e_k \varepsilon_m = (m+(\alpha -2)k) \varepsilon_{m+k} + c_0 k^3 \theta_{m+k}^{(0)} + c_1 k^2 \theta_{m+k}^{(1)}
\]

The representation $\rho$ gives no added restrictions on our coeficients. All we need to do is find an element in the 1-dimensional non-trivial solution space to find the equivalence class of non-trivial cocycles. The trivial cocycle of degree 2 is
$
(\alpha-1) k^3 \theta_{k+m}^{(0)} - (2\alpha -3) k^2 \theta_{k+m}^{(1)}
$.
For $c_0=1$, and $c_2 = -2$, the cocycle is not trivial. Thus we obtain an equivalence class of $ k^3  \theta^{(0)} -  2k^2  \theta^{(1)} $ for $\alpha - \beta =2$.

\subsubsection{Cocycles of degree 4}

The case of $n=4$ is very similar to the previous case. For non-trivial cocycles, $\alpha - \beta = 3$, and $c_{3} =0$. Then by (\ref{epsilon}),
\[
 e_k \varepsilon_m = (m+(\alpha -3)k) \varepsilon_{m+k} +c_0 k^4 \theta_{m+k}^{(0)}+ c_1 k^3 \theta_{m+k}^{(1)} + c_2 k^2 \theta_{m+k}^{(2)} 
\]

The representation $\rho$ gives the following condition from  (\ref{rhorep}) when $i=2, j=3$:
\[
 12(\alpha -1)c_1 + 6(\alpha-1)c_2 + 24c_0 =0
 \]
Thus cocycles are of the form $-\frac{(\alpha-1)}{4}( 2c_1 +c_2) k^4 \theta_{m+k}^{(0)} + c_1 k^3 \theta_{m+k}^{(1)} + c_2 k^2 \theta_{m+k}^{(2)}$. The 1-dimensional trivial solution space is given by the equivalence class of 
\[
(\alpha-1) k^4 \theta_{m+k}^{(0)} -(3\alpha - 4) k^3 \theta_{m+k}^{(1)} + 6(\alpha -2) k^2 \theta_{m+k}^{(2)}
\]
Then the 1-dimensional non-trivial solution space is spanned by the cocycle where $2c_1 = -c_2$. Thus, we obtain the equivalence class of $k^3 \theta^{(1)} - 2k^2 \theta^{(2)}$, for $\alpha - \beta =3$.

\subsection{Cases when $n \geq 5$}

As soon as $n \geq 5$, Part (2) of Lemma \ref{AB} can be used. If we set $i=2$, then
\[
(j+1)!(2-j)c_{n-j-1} - 2(-1)^{j-1}(j\alpha - n+1)c_{n-2} - j!(2\alpha- n+j-1)c_{n-j}=0
\]
for $j=3, 4, \dots, n-2$ so that we will obtain $n-3$ formulas for $ n \geq 5$. If we set $i=3$, 
\[
(j+2)!(3-j)c_{n-j-2} - 6(-1)^{j-1}(j\alpha - n+2)c_{n-3} + j!(3\alpha- n+j-1)c_{n-j}=0
\]
for $j=4, 5, \dots, n-2$ so that we obtain $n-5$ formulas for $ n \geq 7$. From this, we obtain the following coefficient matrix, where each column $j$ corresponds to the coefficient $c_{j-1}$.

$$
\begin{bmatrix}
 0 & 0 & 0 & \cdots & 0 & 0 & 0 & a_{1,n-4} & a_{1,n-3} & a_{1,n-2} \\
 0 & 0 & 0 &\cdots& 0 & 0 & a_{2,n-5} & a_{2,n-4} & 0 & a_{2,n-2} \\
 0 & 0 & 0 &\cdots& 0 & a_{3,n-6} & a_{3,n-5} & 0 & 0 & a_{3,n-2} \\
 \vdots&\vdots&\vdots& &\vdots&\vdots&\vdots&\vdots&\vdots&\vdots\\
 0 & a_{n-4,1} & a_{n-4,2} &\cdots& 0 & 0 &  0 & 0 & 0 & a_{n-4,n-2} \\
 a_{n-3,0} & a_{n-3,1} & 0 &\cdots& 0 & 0 &  0 & 0 & 0 & a_{n-3,n-2} \\
 0 & 0 & 0 &\cdots& 0 & b_{1,n-6} & 0 & b_{1,n-4} & b_{1,n-3} & 0 \\
 0 & 0 & 0 &\cdots& b_{2,n-7} & 0 & b_{2,n-5} & 0 & b_{2,n-3} & 0 \\
\vdots&\vdots&\vdots&  &\vdots&\vdots&\vdots&\vdots&\vdots&\vdots\\
 0 & 0 & b_{n-7,2} &\cdots& 0 & 0 &  0 & 0 & b_{n-7,n-3} & 0 \\
 0 & b_{n-6,1} & 0 &\cdots& 0 & 0 &  0 & 0 & b_{n-6,n-3} & 0 \\
 b_{n-5,0} & 0 & b_{n-5,2} &\cdots& 0 & 0 &  0 & 0 & b_{n-5,n-3} & 0 \\
\end{bmatrix}
$$

where, for $1 \leq i \leq n-3$,
\begin{align*}
 a_{i,n-2} &= 2(-1)^{i+2}((i+2)\alpha - (n-1))\\ 
 a_{i,n-2-i} &= -(i+2)!(2\alpha-(n-i-1)) \\ 
 a_{i,n-3-i} &= -i(i+3)!
\end{align*}

and for $1 \leq i \leq n-5$,
\begin{align*}
 b_{i,n-3} &= 6(-1)^{i+3} ((i+3)\alpha-(n-2))\\ 
 b_{i, n-3-i} &= (i+3)!(3\alpha - (n-i-2))\\ 
 b_{i,n-5-i} &= -i(i+5)!
\end{align*}

This matrix encodes a system of $(n-3) + (n-5) = 2n -8$ equations in $n-1$ variables, our $n-1$ coefficients. We are guaranteed to always have a 1-dimensional trivial solution space, which means that the rank of our matrix is at most $n-2$. If there exists a non-trivial solution, then the rank would have to be less than or equal to $n-3$.

To demonstrate this, we can further reduce this matrix to

$$
\begin{bmatrix}
 0 & 0 & 0 &\cdots& 0 & 0 & 0 & a_{1,n-4} & a_{1,n-3} & a_{1,n-2} \\
 0 & 0 & 0 &\cdots& 0 & 0 & a_{2,n-5} & 0 & a_{2,n-3} & a'_{2,n-2} \\
 0 & 0 & 0 &\cdots& 0 & a_{3,n-6} & 0 & 0 & a_{4,n-3} & a'_{3,n-2} \\
 \vdots&\vdots&\vdots&  &\vdots&\vdots&\vdots&\vdots&\vdots&\vdots\\
 0 & a_{n-4,1} & 0 &\cdots& 0 & 0 &  0 & 0 & a_{n-4,n-3} & a'_{n-4,n-2} \\
 a_{n-3,0} & 0 & 0 &\cdots& 0 & 0 &  0 & 0 & a_{n-3,n-3} & a'_{n-3,n-2} \\

 0 & 0 & 0 &\cdots& 0 & 0 & 0 & 0 & b'_{1,n-3} & b_{1,n-2} \\
 0 & 0 & 0 &\cdots& 0 & 0 & 0 & 0 & b'_{2,n-3} & b_{2,n-2} \\
 \vdots&\vdots&\vdots&  &\vdots&\vdots&\vdots&\vdots&\vdots&\vdots\\
 0 & 0 & 0 &\cdots& 0 & 0 &  0 & 0 & b'_{n-7,n-3} & b_{n-7,n-2} \\
 0 & 0 & 0 &\cdots& 0 & 0 &  0 & 0 & b'_{n-6,n-3} & b_{n-6,n-2} \\
 0 & 0 & 0 &\cdots& 0 & 0 &  0 & 0 & b'_{n-5,n-3} & b_{n-5,n-2} \\
\end{bmatrix}
$$

where, for $2 \leq i \leq n-3$,
\begin{align*}
 a_{1,n-2} =& -2(3\alpha - (n-1))\\
  a_{1,n-3} =& -6(2\alpha-(n-2)) \\ 
  a_{1,n-4} =& -24 \\
 a'_{i,n-2} =& \frac{2(-1)^i}{(i-1)!} \Bigg(\left[\sum_{j=3}^{i+1}(j-3)!(j\alpha-(n-1))\prod_{\ell=j}^{i+1}(2\alpha-(n-\ell))\right] \Bigg. \\ & \Bigg. +(i-1)!((i+2)\alpha - (n-1))\Bigg)\\ 
 a_{i,n-3} =& \frac{6(-1)^i}{(i-1)!}\prod_{j=1}^{i}(2\alpha - (n-(j+1))) \\ 
 a_{i,n-3-i} =& -i(i+3)!
\end{align*}
and for $1 \leq i \leq n-5$,
\begin{align*}
 b_{i,n-2} =&\left(  \dfrac{( 3\alpha - (n-2-i))}{i} \right)a'_{i,n-2} - \left( \dfrac{i}{i+2}\right) a'_{i+2, n-2}\\
 b'_{i,n-3} =& 6(-1)^{i+3}((i+3)\alpha - (n-2)) - \left(\dfrac{i}{i+2}\right)a_{i+2,n-3} + \left( \dfrac{(3\alpha - (n-2-i))}{i} \right)a_{i,n-3}
\end{align*}
which gives us the explicit formulae for $2 \leq i \leq n-5$
\begin{align*}
 b_{1,n-2} &= \dfrac{1}{3}(-n^3+7\alpha n^2+3n^2-16\alpha^2n-9\alpha n-2n +12\alpha^3+6\alpha) \\
 b'_{1,n-3} &= -n^3  +6\alpha n^2+ 3n^2  -12 \alpha^2 n -6\alpha n -2n +8 \alpha^3 +4\alpha \\
 b_{i,n-2} &= \frac{2(-1)^i}{i!} \Bigg[ \dfrac{i}{(i+1)(i+2)} \left(\sum_{j=1}^{i+3}(j-3)!(j\alpha-n+1)\prod_{\ell=j}^{i+3}(2\alpha - (n-\ell)) \right) \Bigg.\\
 & +(3\alpha - n +2+i) \left( \sum_{j=1}^{i+1}(j-3)!(j\alpha-(n-1))\prod_{\ell=j}^{i+1}(2\alpha - (n-\ell)) \right) \\
 & (i-1)!((i+2)\alpha -n+1)(3\alpha -n +2+i) + \dfrac{i(i+1)!}{(i+2)(i+1)}((i+5)\alpha -n+1) \Bigg] \\
 b'_{i,n-3} &= \dfrac{6(-1)^{i+1}}{(i+2)!} \Bigg[ (i+2)!\left( (i+3)\alpha - n+2) +  \prod_{l=j}^{i} (2\alpha - n+j+1) \right) \Bigg. \\
 & \Bigg. \times \left( i(2\alpha - n +i+2)(2 \alpha -n +i+3) - (i+1)(i+2)(3 \alpha -n +2+i) \right) \Bigg] 
 \end{align*}

\begin{remark}
This last matrix shows the rank is always at least $n-3$, which means there is always at most a 1-dimensional non-trivial solution space.
\end{remark}

\subsubsection{Cocycles of degree 5}

When $n=5$, we only get two formulae and thus our complete coefficient matrix is only a $2 \times 4$ matrix, so the rank is automatically less than or equal to $n-3=2$. Thus there exists a non-trivial solution, and we simply solve to find a satisfactory cocycle not equivalent to the trivial cocycle, as done in the previous sections. This gives an equivalence class of:
\[
 (\alpha -1) k^5 \theta^{(0)} - k^4 \theta^{(1)} - 12 k^3 \theta^{(2)} + 24k^2 \theta^{(3)}, \text{ for } \alpha - \beta =4.
\]

\subsubsection{Cocycles of degree 6}

When $n = 6$, the rank of the coefficient matrix will only be $n-3$ when $b'_{1,n-3} = b_{1,n-2} = 0$. This means there is only a non-trivial solution where both polynomial coefficients (\ref{2}) and (\ref{1}) are zero at the same time.
\begin{equation}\label{2}
 b'_{1,n-3} = -n^3  +6\alpha n^2+ 3n^2  -12 \alpha^2 n -6\alpha n -2n +8 \alpha^3 +4\alpha
\end{equation}
\begin{equation}\label{1}
 b_{1,n-2} = \dfrac{1}{3}(-n^3+7\alpha n^2+3n^2-16\alpha^2n-9\alpha n-2n +12\alpha^3+6\alpha)
\end{equation}
Solving for the roots of these polynomials, we get two shared roots of the form $\alpha = \frac{n \pm \sqrt{-2+3n}}{2} $ which reduces to $\alpha = 5$ and $\alpha=1$. Similar to the methods used in the previous section, there is a 1-dimensional non-trivial solution space given by the equivalence classes of:
\begin{align*}
 & k^5 \theta^{(1)} + 5 k^4 \theta^{(2)} + 30 k^3 \theta^{(3)} - 60 k^2 \theta^{(4)}, \text{ for } \alpha =5, \beta =0. \\
&6 k^6 \theta^{(0)} - 11 k^5 \theta^{(1)} + 5 k^4 \theta^{(2)} + 30 k^3 \theta^{(3)}- 60 k^2 \theta^{(4)}, \text{ for } \alpha =1, \beta =-4.
\end{align*}

\subsubsection{Cocycles of degree 7}

When $n =7$, the rank of the coefficient matrix will equal $n-3$ when $b'_{1,n-3} = b_{1,n-2} = b'_{2,n-3} = b_{2,n-2} = 0$. This means there is only a non-trivial solution space when all four polynomial are zero at the same time.

Solving for the roots of $b'_{1,n-3}, b_{1,n-2}, b'_{2,n-3}$ and $b_{2,n-2}$, all four have the common roots of $\alpha = \frac{n \pm \sqrt{-2+3n}}{2}$. This gives 2 cases of non-trivial cocycles
\begin{align*}
 \text{For } \alpha = \dfrac{7 + \sqrt{19}}{2}:& -\dfrac{22+\sqrt{19}}{4} k^7 \theta^{(0)}+ \dfrac{31+7\sqrt{19}}{2}k^6 \theta^{(1)} - (25+7\sqrt{19})k^5 \theta^{(2)}
 \\ 
 & + 30 k^4 \theta^{(3)} + 120 k^3 \theta^{(4)} - 240 k^2 \theta^{(5)} 
 \\
 \text{For }\alpha = \dfrac{7 - \sqrt{19}}{2}: &-\dfrac{22-\sqrt{19}}{4} k^7 \theta^{(0)}+ \dfrac{31-7\sqrt{19}}{2}k^6 \theta^{(1)} - (25-7\sqrt{19}) k^5 \theta^{(2)}
 \\
 &  +30 k^4 \theta^{(3)} + 120 k^3 \theta^{(4)} - 240 k^2 \theta^{(5)} 
\end{align*}

\subsubsection{Cocycles of degree greater than or equal to 8}

When $n \geq 8$, the system will only admit non-trivial cocycles when 
\[b'_{1,n-3} = b_{1,n-2} = b'_{2,n-3} = b_{2,n-2} = b'_{3,n-3} = b_{3,n-2} = \cdots = b'_{n-5,n-3} = b_{n-5,n-2} = 0
\]
 If there exists a common solution to these formulae, then this solution must satisfy all of the first four formulae, thus it necessarily must be of the form $\alpha =  \frac{n \pm \sqrt{-2+3n}}{2}$. It is easy to check that this is not a solution to $b'_{3,n-3}$ or $b_{3,n-2}$. This means the rank of the matrix is always at least $n-2$ and thus this case never admits a non-trivial solution space.

\section{Delta Cocycles} 

For this section, we assume that the conditions for delta function cocycles hold, i.e. $\alpha=0$ and $\gamma =0$. These cocycles will classify extensions of the form
\[
0 \rightarrow D(0) \rightarrow M \rightarrow T(\beta, 0) \rightarrow 0
\]
 As $D(0) \subset T(0,0)$, this extension will admit an extension
\[
0 \rightarrow T(0,0) \rightarrow M' \rightarrow T(\beta, 0) \rightarrow 0
\]
In this section, we will find these delta cocycles in the context $M'$ rather than $\widetilde{M}_N$, where $\delta_k(t^m) = \delta_{k+m,0}v_{m+k}$.
These delta cocycles will be of the form 
\[
\tau(k,m) = \delta_{m+k,0} f(k,m)= \delta_{m+k,0} f(k,-k)
\]
Thus, $f$ is a function of $k$. In fact, by Theorem \ref{theorem} for $e_k . w_{-k}$ and $\beta \neq 1$, $f(k)$ must be a polynomial function in $k$. As we will deal with the special case of $\beta =1$ in the next section, for now we will assume that $f(k)$ is a polynomial in $k$ for every $\beta$.

\subsection{Conditions for delta functions}

As derived in Section 2, cocycles must satisfy the relation (\ref{taucondition}) so for $\tau(k,m) = \delta_{k+m, 0} f(k)$,
this reduces to:
\begin{equation}\label{validdelta}
 (s-k) f(k+s) = ((\beta-1)s-k) f(k) - ((\beta-1)k-s) f(s)
\end{equation}

A trivial cocycle is a 1-coboundary which becomes the zero function under some change of basis. By the change of basis $u_m = w_m + c \delta_{m,0}v_0$, we obtain that $\tau(k,m) = k (1-\beta) \delta_{k+m,0}$ is the trivial cocycle. 

To solve for cocycles $\tau(k,m) = \delta_{k+m,0}f(k)$, we know that the function $f(k)$ must be polynomial so that for some $n\in \mathbb{N}, f(k)$ is given by
\[
 f(k) = a_n k^n + a_{n-1} k^{n-1} + \cdots + a_1 k + a_0.
\]

Equation (\ref{validdelta}) is homogeneous in $s$ and $k$ so each case $f(k) = k^m$ can be treated separately. 
As in the previous section, we simply solve the system of coefficients we obtain from (\ref{validdelta}) and find a non-trivial solution to obtain the equivalence classes described below.
\begin{align}
\delta_{k+m,0},& \text{ when } \beta =1\\
\delta_{k+m,0} k,& \text{ when } \beta=1\\
\delta_{k+m,0} k^2 ,& \text{ when } \beta=0\\
\delta_{k+m,0} k^3 ,& \text{ when } \beta=-1
\end{align}

When $n \geq 4$, by (\ref{validdelta}) we can obtain the system
\[
0= (s-k)(k+s)^n - ((\beta-1)s-k)k^n + ((\beta-1)k-s)s^n
\] 
The coefficient at $s^{n-1} k^2$ for $n \geq 4$ will be given by:
\begin{align}
 \left(  { n \choose 2 } - {n \choose 1} \right) &= \dfrac{1}{2} n(n-1) -n \neq 0 \text{ when } n \neq 3
\end{align}

Thus, this coefficient is non-zero for any value of $\beta$. Therefore this system will only admit the trivial cocycle as a solution and there are no other cocycles of this form.

\section{The case of $\beta=1$}

The approach used in Section 2 to find conditions on our cocycles depended on the fact that $\beta \neq 1$ as 
there was no obvious way in which to define $\widetilde{M}_N$. In particular, $\varepsilon_j \notin \widehat{T(1, \gamma)}$ for any $j \in \gamma + \mathbb{Z}$. We deal with this case here, using a similar approach by defining a basis of $M$ using $\widehat{M}$.

As we have already found all possible polynomial cocycles, we will now assume that cocycles are strictly non-polynomial. 

\subsection{The structure of non-polynomial cocycles}

 First, we would like to find an appropriate basis of $M$. Introduce $\sigma_0 \in \widehat{M}$ as $ \sigma_0 = \phi(e_0, w_{\gamma})$ and set $\widehat{w}_{m+\gamma} = \frac{1}{m+\gamma} \sigma_0 (t^m)$ for $m+\gamma \neq 0$ and let $\widehat{w}_0 = w_0$. Then $\widehat{w}_{m+\gamma} = w_{m+\gamma} + \frac{\tau(m,\gamma)}{m+\gamma} v_{m+\gamma}$ for all $k \in \mathbb{Z}$ and $m \in \mathbb{Z}$ such that $m+\gamma \neq 0$. 

By Theorem \ref{theorem}, $\widehat{M} \cong A \otimes \widehat{M}_\gamma$ for some finite representation $\widehat{M}_\gamma$ of $\mathscr{L}_+$. Then $\sigma_0(t^m) = t^{m} \otimes \sigma$ for some $\sigma \in \widehat{M}_\gamma$. For $m+\gamma \neq 0$,
\begin{align*}
e_k  \widehat{w}_{m+\gamma} 
&= (m+k+\gamma) \widehat{w}_{m+k+\gamma} + \frac{1}{m+\gamma} \left(ku_1 (t^{m+k}) + \cdots + k^nu_n(t^{m+k}) \right) 
\end{align*}
for all $k \in \mathbb{Z}$, where $u_i = \frac{1}{i!} \rho \left( z^i \frac{d}{dz} \right) \sigma$ as in Theorem \ref{theorem}. 

Notice that $(e_k \widehat{\pi}(\sigma_0))(t^m) = 0, \, \forall m \in \mathbb{Z}$ such that $m+\gamma \neq 0$. It follows from $e_k \phi(e_0, \overline{w}_{\gamma}) = 0 $ that $ \phi(e_0, \overline{w}_{\gamma}) =0$ and so $\widehat{\pi}(\sigma_0) \in \widehat{M}(\alpha,\gamma)$. 
This shows that $k u_1 + \cdots + k^n u_n \in \widehat{M}(\alpha,\gamma)$, so that for the cocycle $\widetilde{\tau}(k,m) = \frac{1}{m} (k u_1+ \cdots + k^n u_n)$ defined on $m \in \gamma +\mathbb{Z}$ non-zero, $m \widetilde{\tau}(k,m)$ is contained in a submodule of the form described in Proposition \ref{Prop}. In other words, the cocycle $\widetilde{\tau}(k,m)$ will be a polynomial function or a $\delta_{m+k,0}$ function when $\alpha=0$, with a factor of $m^{-1}$. Here we may apply Theorem \ref{theorem} to obtain that if a cocycle is given by $\frac{1}{m}\delta_{m+k,0}f(k)$, $f(k)$ is polynomial in $k$.

The cocycle $\widehat{\pi}(\widetilde{\tau})=\tau'(k,m)$ on $M$ is not defined when $m=0$. This $\tau$-function completely determines the cocycle on the short exact sequence with $T^\circ(1)$. To extend this cocycle onto $M$, we introduce a possible $\delta_{m,0}$ function to determine the cocycle on the basis vector $w_0$.
\[
e_k w_0 = k \widehat{w_k} + \mu(k)v_{k}
\]

We can extend the cocycle to the whole space of $M$ by setting $\tau(k,m)$ to be the piecewise function:
\[
\tau(k,m) = \left\lbrace \begin{array}{ll} \tau'(k,m), & m \neq 0 \\ \mu(k),&m=0 \end{array} \right.
\]
These two components are independent thus we may consider each case separately. 
Hence, for $\beta =1$, cocycles in $M$ can be polynomial functions in $k$ with a possible factor of $m^{-1}$, $\delta_{k+m,0}$ functions with a possible factor of $m^{-1}$ or $\delta_{m,0}$ functions. 

\subsection{Cocycles with a factor of $m^{-1}$}

Suppose that $\tau'(k,m) = m^{-1} \mu(k,m)$ for $k \in \mathbb{Z}, m \in \gamma + \mathbb{Z}$. If this is a cocycle, (\ref{taucondition})
reduces this to the case that $\mu(k,m)$ is a cocycle for $\beta =0$. 

In the case that $\alpha=0$ and $\mu(k,m) = \delta_{k+m, 0}f(k)$ there are two cocycles for $\beta=0$: the trivial cocycle $\delta_{k+m,0}k$ and the nontrivial cocycle $\delta_{k+m,0}k^2$.
This first case gives a cocycle of the form $\delta_{k+m,0}$ which was found to be non-trivial in Section 4.
The second case is exactly the equivalence class of $\delta_{k+m,0} k$ found in Section 4.

In the case that $\mu(k,m)$ is a polynomial function, then we would like to see when this cocycle is a coboundary. In other words, there exists $g: \gamma + \mathbb{Z} \rightarrow \mathbb{C}$ such that 
\[
g(m) (m+\alpha k) - g(m+k) (m+k) =  \frac{1}{m} \mu(k,m)
\]
But if $kg(k) = f(k)$, then this reduces to 
\begin{align*}
f(m) (m+ \alpha k) - f(m+k) m &= \mu(k,m)
\end{align*}
which is the same condition as $\mu(k,m)$ equivalent to the trivial polynomial cocycle for $\beta =0$. Thus, if $\frac{1}{m} \mu(k,m)$ is a non-trivial cocycle  for $\beta=1$, $\mu(k,m)$ must be a non-trivial cocycles for $\beta=0$.

In Section 6, we describe the polynomial equivalence classes of cocycles in $M$. Using Table 6.2, there are 7 possible forms of the function $\mu(k,m)$.
Most of these cases will reduce to equivalence classes already known. Three cases admit new cocycles, giving the equivalence classes
\[
\begin{array}{lc}
      \alpha =0, \beta =1, & m^{-1}k \\ 
      \alpha =1, \beta =1, & m^{-1} k^2 \\
      \alpha =2, \beta =1, &  m^{-1} k^3 +k^2
\end{array}
\]

\begin{remark} These functions are not polynomial but can be considered such under some change of basis.

Consider the homomorphism from $T(0,\gamma) \rightarrow T(1, \gamma)$ by $v_m \rightarrow m w_m$. For $\gamma \notin \mathbb{Z}$, this map is bijective so that $T(0,\gamma) \cong T(1,\gamma)$. In the case that $\gamma \in \mathbb{Z}$, the image of this map is $T^\circ (1)$ and the kernel is $\mathbb{C}v_0=D(0)$ so that
\[
T(0,0)/D(0) \cong T^\circ(1)
\]
But as these $\tau'$-functions are zero at $m=0$, they can be considered cocycles defined on $T^\circ(1)$.
By applying this map to $w_m \in T(1, \gamma)$ to $w_m' \in T(0, \gamma)$,
\begin{align*}
e_k. w_m &= (k+m)w_{m+k} + \frac{1}{m}\mu(k,m) v_{m+k} \\
m e_k .w_m &= m (k+m)w_{m+k} + \mu(k,m) v_{m+k} \\
e_k .w_m' &= m w_m' + \mu(k,m) v_{m+k}
\end{align*}
and $\mu(k,m)$ is a polynomial cocycle. Although these module extensions are not isomorphic, these cocycles can be thought of as polynomials in some sense.
\end{remark}

\subsection{Delta cocycles}

As in Section 4, we will consider delta functions of the form $\delta_{m,0}f(k,m)$. This reduces to $f(k,m) = \mu(k)$ where $\mu$ is polynomial in $k$. If we take the change of basis $w_0 ' = w_0 + v_0$, then we find that $\delta_{m,0} \alpha k$ is the trivial cocycle.

Now, we take $\tau(k,m) = \delta_{m,0} \mu(k)$ 
so by the cocycle condition (\ref{taucondition}),
\begin{equation}\label{delta}
 (k-s) \mu(k+s) + (k+\alpha  s) \mu(k) - (s+\alpha k) \mu(s)=0
\end{equation}

Using this condition, we may derive an infinite system of equations on $\mu(k)$ for $k \in \mathbb{Z}$. Most importantly, we obtain the following 5 conditions:
\[\begin{array}{l}
 \alpha \mu(0) =0 \\
 \text{ } \; \mu(3) = (2+\alpha) \mu(2) - (1+2 \alpha)\mu(1)\\ 
 \text{ } \; \mu(4) = \frac{1}{2} (3 + \alpha)(2+\alpha) \mu(2) - (2+5\alpha + \alpha^2) \mu(1) \\ 
 \text{ } \; \mu(5) = \frac{1}{6} (4+\alpha)(3+\alpha)(2+\alpha) \mu(2) - \frac{1}{3} (9 + 26 \alpha + 9 \alpha^2 + \alpha^3) \mu(1) \\ 
 \text{ } \; \mu(5) = (4+4\alpha + 2\alpha^2) \mu(2) - (3 + 8 \alpha + 4 \alpha^2) \; \mu(1)
 \end{array}
 \]
 
The last two of these equations give us $(2\alpha - 3 \alpha^2 + \alpha^3 ) \mu(2) = 2 (2\alpha - 3 \alpha^2 + \alpha^3 ) \mu(1)$.
 As $2\alpha - 3 \alpha^2 + \alpha^3 = \alpha(\alpha-1)(\alpha-2)$, then $\mu(2) = 2\mu(1)$ as long as $\alpha \neq 0,1,2$. 

The proof of the next four lemmas are very similar. Consequently, we will only include the proof of the first one.

\begin{lemma}
 If $\alpha \neq 0,1,2$, then $\mu(n) = n \mu(1)$ for $n \geq 3$.
\end{lemma}

\begin{proof}
 If $n=3$, then $\mu(3) = (2(2+\alpha) -(1+2\alpha) )\mu(1) = 3\mu(1)$.

By induction, set $s=n, k=1$, then
\begin{align*}
 \mu(n+1) &= \frac{1}{n-1} \left( (n+\alpha) \mu(n) - (1+n \alpha) \mu(1) \right)\\
&= \frac{1}{n-1} \left( (n+\alpha) n- (1+n \alpha) \right)\mu(1) \\
&= \frac{1}{n-1} \left( n^2- 1 \right)\mu(1) \\
&= (n+1)\mu(1)
\end{align*}
\end{proof}

Thus by the above recurrence relation $\mu(k) = k \mu(1)$ except when $\alpha =0, 1$ or $2$. The next lemma considers $\mu(k)$ in these three special cases. 

\begin{lemma}
\begin{enumerate}
\item If $\alpha = 0$, then $\mu(n) = (n-1)\mu(2) - (n-2) \mu(1)$ for $n \geq 3$ and $\mu(k)$ will be a polynomial of degree 1. 
\item  If $\alpha = 1$, then $\mu(n) = \frac{n(n-1)}{2}\mu(2) - (n^2 - 2n) \mu(1)$ for $n \geq 3$ and $\mu(k)$ is a polynomial of degree 2.
\item  If $\alpha = 2$, then $\mu(n) = \frac{n^3 -n}{6}\mu(2) - \frac{n^3 - 4n}{3} \mu(1)$ for $n \geq 3$ and $\mu(k)$ is a polynomial of degree $3$.
\end{enumerate}
\end{lemma}

Thus if $\mu(k)$ is non-trivial, then it has possible values of:
\[
\begin{array}{ll}
      \alpha=0, & \mu(k) =k \text{ or } 1\\  
      \alpha =1, & \mu(k) = k^2 \\ 
      \alpha=2, & \mu(k) = k^3 
      \end{array}
      \]
which are exactly the dual delta function cocycles we found in Section 4.

\section{Conclusion}

A summary of the results are given in Table 1. 
Table 6.2 describes all cocycles for the original extension. By directly applying the map $\pi: \widetilde{M}_N \rightarrow M$, we do not obtain exactly the same cocycle, only a cocycle that is in the same equivalence class. This will not change the equivalence class though, so that we can shift the representative of the class to a function of the same form. The choice of these equivalence classes are a bit arbitrary; they are modelled after the results of Feigin and Fuks \ref{FF}. There is a typo in \ref{FF} in the case of the degree 7 cocycles, otherwise we obtain the same results.

\begin{table}[htbp]
 \centering
 \caption{Polynomial cocycles in $\widetilde{M}_N$}
 \renewcommand{\arraystretch}{1.4}
 \small
  \begin{tabular}{|l|r|c|}\hline
   $\alpha -\beta =0$ & $n=1$ & $k \theta^{(0)}$        \\ \hline 
   $\alpha=1,\beta=0$ & $n=2$ & $k \theta^{(1)}$        \\ \hline
   $\alpha=1,\beta=0$ &$n=2$ &  $k^2 \theta^{(0)} $     \\ \hline
   $\alpha - \beta =2$ &$n=3$ &  $k^3 \theta^{(0)} - 2 k^2 \theta^{(1)}$                                                       \\ \hline
   $\alpha - \beta =3$ &$n=4$ &  $k^3 \theta^{(1)} - 2 k^2 \theta^{(2)}$\\ \hline
   $\alpha - \beta =4$ &$n=5$ &  $(\alpha-1) k^5 \theta^{(0)} -k^4 \theta^{(1)} - 12k^3 \theta^{(2)} + 24 k^2 \theta^{(3)}$   \\ \hline
   $\alpha=1,\beta=-4$ &$n=6$ &  $ 2 k^5 \theta^{(1)} + 10 k^4 \theta^{(2)} + 60 k^3 \theta^{(3)} - 120 k^2 \theta^{(4)}$ \\ \hline
   $\alpha=5,\beta=0$ &$n=6$ &  $12 k^6 \theta^{(0)} - 22 k^5 \theta^{(1)} + 10 k^4 \theta^{(2)} + 60 k^3 \theta^{(3)}- 120 k^2 \theta^{(4)}$ \\ \hline
   $\alpha=\frac{7+\sqrt{19}}{2},$ & $n=7$ & $-\frac{22+\sqrt{19}}{4} k^7 \theta^{(0)}+ \frac{31+7\sqrt{19}}{2}k^6 \theta^{(1)}- (25+7\sqrt{19})k^5 \theta^{(2)} $ \\
  $\beta=-\frac{5+\sqrt{19}}{2}$ & & $  + 30 k^4 \theta^{(3)} + 120 k^3 \theta^{(4)} - 240 k^2 \theta^{(5)} $ \\ \hline
    $\alpha=\frac{7-\sqrt{19}}{2},$ & $n=7$ & $-\frac{22-\sqrt{19}}{4} k^7 \theta^{(0)}+ \frac{31-7\sqrt{19}}{2}k^6 \theta^{(1)} - (25-7\sqrt{19}) k^5 \theta^{(2)}$ \\
  $\beta=-\frac{5-\sqrt{19}}{2}$ & & $  +30 k^4 \theta^{(3)} + 120 k^3 \theta^{(4)} - 240 k^2 \theta^{(5)} $ \\ \hline
 \end{tabular}
\end{table}

Most extensions admit classes of cocycles for a range of $\alpha$ and $\beta$. The few exceptions occur in pairs; these pairs are module extensions that are dual to each other in the sense that for the extension ${M}/{T(\alpha,\gamma)} \cong T(\beta, \gamma)$, the dual extension is given by ${M'}/{T(1-\beta, -\gamma)}\cong T(1-\alpha,-\gamma)$ with the dual cocycle $\tau^*(k,m) = \tau(k,-m-k)$. 

\begin{table}[htbp]\label{second}
 \centering
 \caption{Polynomial cocycles in $M$}
 \renewcommand{\arraystretch}{1.4}
 \small
  \begin{tabular}{|l|r|c|}\hline
   $\alpha -\beta =0$ & $n=1$ & $k $        \\ \hline 
   $\alpha=1,\beta=0$ & $n=2$ & $k m$        \\ \hline
   $\alpha=1,\beta=0$ &$n=2$ &  $k^2 $     \\ \hline
   $\alpha - \beta =2$ &$n=3$ &  $k^3  + 2 k^2m$  \\ \hline
   $\alpha - \beta =3$ &$n=4$ &  $ k^3m + k^2m^2  $\\ \hline
   $\alpha - \beta =4$ &$n=5$ &  $(\alpha-4) k^5 +k^4m- 6k^3m^2  - 4 k^2 m^3    $   \\ \hline
   $\alpha=1,\beta=-4$ &$n=6$ &  $ 12 k^6 +22 k^5m + 5k^4m^2  -10 k^3m^3 - 5k^2m^4 $ \\ \hline
   $\alpha=5,\beta=0$ &$n=6$ &  $2k^5 m - 5 k^4m^2 + 10 k^3m^3 + 5 k^2m^4$ \\ \hline
   $\alpha=\frac{7+\sqrt{19}}{2},$ & $n=7$ & $-\frac{22+\sqrt{19}}{4} k^7 - \frac{31+7\sqrt{19}}{2}k^6 m$ \\
  $\beta=-\frac{5+\sqrt{19}}{2}$\hspace{1pt} & & $  - \frac{25+7\sqrt{19}}{2} k^5m^2 -5 k^4 m^3 + 5 k^3m^4 + 2 k^2m^5 $ \\ \hline
    $\alpha=\frac{7-\sqrt{19}}{2},$ & $n=7$ & $-\frac{22-\sqrt{19}}{4} k^7 - \frac{31-7\sqrt{19}}{2}k^6 m$ \\
 $\beta=-\frac{5-\sqrt{19}}{2}$\hspace{1pt} & & $  - \frac{25-7\sqrt{19}}{2} k^5 m^2 -5 k^4 m^3 + 5 k^3m^4 + 2 k^2 m^5 $ \\ \hline
 \end{tabular}
\end{table}

These dual modules are also useful in finding delta functions. In fact, all delta function cocycles can be obtained from the work done in Section 3 as every delta cocycle obtained in Section 5 is dual to a cocycle found in Section 4.

\begin{table}[htbp]
 \centering
 \caption{Non-polynomial cocycles}
 \renewcommand{\arraystretch}{1.4}
 \small
   \begin{tabular}{|c|c|c|}\hline
   $\alpha =0, \beta=-1$ & $\gamma\in \mathbb{Z}$ & $\delta_{k+m,0} k^3$     \\ \hline
   $\alpha =0, \beta=0$ & $\gamma\in \mathbb{Z}$ & $\delta_{k+m,0}k^2$      \\ \hline
   $\alpha =0, \beta =1$ & $\gamma\in \mathbb{C}$ & $m^{-1}k$              \\ \hline
   $\alpha =0, \beta=1$ & $\gamma\in \mathbb{Z}$ & $\delta_{k+m,0} $          \\ \hline
   $\alpha =0, \beta=1$ & $\gamma\in \mathbb{Z}$ & $\delta_{k+m,0} k$      \\ \hline
   $\alpha =0, \beta=1$ & $\gamma\in \mathbb{Z}$ & $\delta_{m,0} k$        \\ \hline
   $\alpha =1, \beta =1$ & $\gamma\in \mathbb{C}$ & $m^{-1} k^2$           \\ \hline
   $\alpha =1, \beta=1$ & $\gamma\in \mathbb{Z}$ & $\delta_{m,0} k^2$        \\ \hline
   $\alpha =2, \beta =1$ & $\gamma\in \mathbb{C}$ & $ m^{-1} k^3 +k^2$   \\ \hline
   $\alpha =2, \beta=1$ & $\gamma\in \mathbb{Z}$ & $\delta_{m,0} k^3$        \\ \hline
   \end{tabular}
\end{table}

Through the method described in the paper, two delta function cocycles of degree zero were found in the case of $\alpha=0$ and $\beta =1$. As shown in Section 3, the sum of these cocycles in a 1-coboundary so that, in fact, these cocycles are in the same equivalence class. 

These delta cocycles are of interest for another reason. Using the delta functions that give an extra factor on $v_0$ in the $W_1$-action on $T(\beta, \gamma)$, i.e. cocycles of the form $\delta_{k+m,0}\mu(k)$, we can construct new Lie algebras. Let $L$ be the Virasoro algebra which is spanned by $\lbrace L(m),c_1\rbrace$ where $c_1$ is the central extension, $m\in \mathbb{N}$. Suppose that module $V$ is a module for the Virasoro algebra with basis $\lbrace W(m), c_3 \rbrace$. By viewing $V$ as a (possibly abelian) Lie algebra, we can look at the semi-direct product of $L$ with $V$ spanned by $\lbrace L(m),W(m),c_1,c_2,c_3\rbrace$ with bracket:
\begin{align*}
  [L(k),L(m)] =& (m-k) L(k+m) + \delta_{k+m,0} \frac{k^3-k}{12} c_1\\
  [L(k), W(m)] =& (m-\beta k)W(k+m) + \delta_{k+m, 0} \mu(k) c_2
\end{align*}
where $V$ is a subalgebra and the bracket of any element with a central extension $c_1, c_2, c_3$ is simply zero. 

In this way, we can construct what is called the $W(2,2)$ algebra \ref{W22}, which corresponds to the cocycle $\delta_{k+m, 0}k^3, \beta =-1$. Similarly, we can construct the twisted Heisenberg-Virasoro algebra \ref{repsTHV} which will correspond to the cocycle $\delta_{k+m,0}k^2, \beta =0$.

We can construct two more algebras of this form, given by the cases of $\delta_{k+m,0}, \beta =1$, $\delta_{k+m,0}k, \beta =1$. As in the previous cases, $V$ can be taken as an abelian algebra or as a Heisenberg algebra. In the case that $V$ is a Heisenberg algebra, this construction is only a Lie algebra if $\beta =0$. 

For the cocycle $\delta_{k+m,0}, \beta =1$, we take $V$ to be an abelian Lie algebra. The resulting algebra is given below, where the bracket with any central element is trivial.
\begin{align*}
  [L(k),L(m)] =& (m-k) L(m+n) + \delta_{k+m,0} \frac{k^3-k}{12} c_1\\
  [L(k), W(m)] =& (m+k) W(k+m) + \delta_{m+k,0} c_2 \\
  [W(k),W(m)] = & 0
\end{align*}

The case of $\delta_{k+m,0}k, \beta =1$ will construct the algebra given by the following, where the the bracket with any central element is trivial. 
\begin{align*}
  [L(k),L(m)] =& (m-k) L(m+n) + \delta_{k+m,0} \frac{k^3-k}{12} c_1 \\
  [L(k), W(m)] =& (m+k) W(k+m) + \delta_{m+k,0} k c_2 \\
  [W(k),W(m)] = & 0
\end{align*}
Notice that if we set $I(m) = mW(m)$, 
\[ [L(k), I(m)] = mI(m+k) - m^2 \delta_{m+k,0} \]
This action is very close to the twisted Heisenberg-Virasoro algebra, which only behaves differently on $W(0)$. 

The main goal of this work was to produce an explicit classification of length two module extensions of the Witt algebra. However, this method is promising in finding a similar classification of length two module extensions of the solenoidal subalgebra (see Definition 2.2 of \ref{cuspidal}) of $W_n$, and may even be used to find a classification of module extensions of this type for $W_n$.

\section*{Acknowledgements}

This research is funded by the Natural Sciences and Engineering Research Council of Canada. I would also like to thank my supervisor, Yuly Billig, whose help was crucial in the completion of this work.

\section*{References} 

\begin{enumerate}[label={[\arabic*]}]
 \item\label{modulispaces} E. Arbarello, C. De Concini, V.G. Kac, C. Procesi, Moduli spaces of curves and representation theory, \textit{Comm. Math. Phys.}, \textbf{117}, 1-36 (1988).

 \item\label{CFT} A.A. Belavin, A. M. Polyakov, A.B. Zamolodchikov, Infinite conformal symmetry in two-dimensional quantum field theory, \textit{Nuclear Physics B}, \textbf{241}, No. 2, 333-380 (1984)

 \item\label{jets} Y. Billig, Jet Modules, \textit{Canad. J. Math}, \textbf{59}, No.4, 712-729 (2007).
 
 \item\label{repsTHV} Y. Billig, Representations of the twisted Heisenberg-Virasoro algebra at level zero, \text{Canadian Mathematical Bulletin}, \textbf{46}, 529-537 (2003).

 \item\label{W_nmodules} Y. Billig, V. Futorny, Classification of simple $W_n$-modules with finite-dimensional weight spaces, arXiv:1304.5458 [math.RT]

 \item\label{cuspidal} Y. Billig, V. Futorny, Classification of simple cuspidal modules for solenoidal Lie algebras, arXiv:1306.5478 [math.RT]

 \item\label{conley} C. H. Conley, Bounded Length 3 Representations of the Virasoro Lie Algebra, \textit{Internat. Math. Res. Notices}, No. 12, 609-628 (2001).

 \item\label{FF} B. L. Feigin, D. B. Fuks, Homology of the Lie algebra of vector fields on the line, \textit{Funct. Anal. Appl.}, \textbf{14}, No.3, 201-212 (1980).

 \item\label{fuks} D.B. Fuks, \textit{Cohomology of Infinite-dimensional Lie Algebras}, Consultants Bureau, New York, 1986.
 
 \item\label{fuchs} J. Fuchs, \textit{Affine Lie algebras and quantum groups}, Cambridge University Press, New York, 1992.

 \item\label{humphrey} James E. Humphreys, \textit{Introduction to Lie algebras and Representation Theory}, Springer-Verlag, New York, 1972.

 \item\label{kac} V. G. Kac, A. K. Raina, N. Rozhkovskaya, \textit{Bombay lectures of highest weight representations of infinite dimensional Lie algebras}, Second Edition, World Scientific Publishing Co. Pte. Ltd, Singapore, 2013.

 \item\label{mp2} C. Martin, A. Piard, Indecomposable modules over the Virasoro Lie algebra and a conjecture of V. Kac, \textit{Commun. Math. Phys.}, \textbf{137}, 109-132 (1991).

 \item\label{mp} C. Martin, A. Piard, Classification of the indecomposable bounded admissible modules over the Virasoro Lie algebra with weightspaces of dimension not exceeding two, \textit{Commun. Math. Phys.}, \textbf{150}, 465-493 (1992).

 \item\label{Mathieu} O. Mathieu, Classification of Harish-Chandra modules over the Virasoro Lie algebra, \textit{Invent. math.}, \textbf{107}, 225-234 (1992).

 \item\label{triangular} R. Moody, A. Pianzola, \textit{Lie Algebras with Triangular Decompositions}, Wiley, New York, 1995.

 \item\label{W22} W. Zhang, C. Dong, $W$-algebra $W(2,2)$ and the vertex operator algebra $L \left(\frac{1}{2},0 \right)\otimes L \left(\frac{1}{2},0 \right)$, arXiv:0711.4624 [math.QA]

 \item\label{THV} X. Zhang, S. Tan, H. Lian, Unitary modules for the twisted Heisenberg-Virasoro algebra, arXiv:1201.1338 [math.RA]

 \item Y. Huang, K. Misra, \textit{Lie algebras, vertex operator algebras and their applications}, American Mathematical Society, Portland, 2007.
\end{enumerate}

\end{document}